\documentclass{amsart}
\usepackage{amsthm}
\usepackage{amsmath}
\usepackage{amssymb}
\usepackage[margin=0.5in]{geometry}
\usepackage{enumerate}
\usepackage{enumitem}
\usepackage{mathrsfs}
\usepackage{float}
\usepackage{graphicx}
\usepackage{mathtools}
\usepackage{hyperref}
\usepackage{caption}
\usepackage{subcaption}
\usepackage{faktor}
\usepackage{xcolor}
\usepackage{amsaddr}

\makeatletter
\renewcommand{\email}[2][]{%
  \ifx\emails\@empty\relax\else{\g@addto@macro\emails{,\space}}\fi%
  \@ifnotempty{#1}{\g@addto@macro\emails{\textrm{(#1)}\space}}%
  \g@addto@macro\emails{#2}%
}
\makeatother

\newtheorem{theorem}{Theorem}
\newtheorem{corollary}[theorem]{Corollary}

\newtheorem{lemma}[theorem]{Lemma}
\newtheorem{definition}[theorem]{Definition}

\newtheorem{remark}[theorem]{Remark}
\numberwithin{theorem}{section}
\numberwithin{equation}{section}

\allowdisplaybreaks

\begin{document}
	\title[End-point estimates unified]{End-point estimates of the totally-geodesic Radon transform on simply connected spaces of constant curvature: A Unified Approach}
	\author[A Deshmukh]{Aniruddha Deshmukh}
	\author[A Kumar]{Ashisha Kumar}
	\address[A Deshmukh, A Kumar]{Department of Mathematics, Indian Institute of Technology Indore, Khandwa Road, Simrol, Indore-453552, Madhya Pradesh, India.}
	\email[A Deshmukh]{aniruddha480@gmail.com, phd2001241001@iiti.ac.in}
	\email[A Kumar]{akumar@iiti.ac.in}
	\begin{abstract}
		In this article, we give a unified proof of the end-point estimates of the totally-geodesic $k$-plane transform of radial functions on spaces of constant curvature. The problem of getting end-point estimates is not new and some results are available in literature. However, these results were obtained independently without much focus on the similarities between underlying geometries. We give a unified proof for the end-point estimates on spaces of constant curvature by making use of geometric ideas common to the spaces of constant curvature, and obtaining a unified formula for the $k$-plane transform of radial functions. We also give some inequalities for certain special functions as an application to one of our lemmata.
	\end{abstract}
	\maketitle
	\section{Introduction}
		\label{IntroductionSection}
		The problem of reconstructing an object from the knowledge of its densities along lines or planes has been of interest for the past century. This study often includes the study of Radon transform, named after the Austrian Mathematician, Johann Radon, who first introduced it in $1917$. An immediate generalization of the Radon transform is the $k$-plane transform, which integrates a given function on $k$-dimensional planes in $\mathbb{R}^n$. Formally, given a function $f: \mathbb{R}^n \rightarrow \mathbb{C}$, its (Euclidean) $k$-plane transform is defined to be the function $R_k f: \Xi_{k} \rightarrow \mathbb{C}$, given by
		$$R_k f \left( \xi \right) := \int\limits_{\xi} f \left( x \right) \mathrm{d}_{\xi}x,$$
		where $\Xi_k$ is the collection of all $k$-dimensional planes in $\mathbb{R}^n$, and $\mathrm{d}_{\xi}x$ is understood to be the $k$-dimensional Lebesgue measure on the plane $\xi$. For $k = n - 1$, the transform is called the Radon transform, while for $k = 1$, it is called the X-Ray transform. Inversion of such transforms is of significant interest, both mathematically and due to its applications.
		
		Owing to a wide range of applications, many researchers have turned their gaze to studying the $k$-plane transforms on non-Euclidean spaces. Helgason is one of the pioneers who has given a tremendous amount of input into studying Radon transforms especially on symmetric spaces. To mention a few of Helgason's works, we refer the reader to \cite{HelgasonBook} and the references therein. Berenstein et al. (\cite{BerensteinKurusa}) and Rubin (\cite{RubinInversion}) have worked on the totally-geodesic Radon transform on constant curvature spaces, where they give range characterizations, and inversion techniques.
		
		In this article, we study the boundedness of the totally-geodesic $k$-plane transforms on spaces of constant curvature. The question of boundedness of $k$-plane transforms is not new, and a lot of work has been done for the Euclidean case. In the early days of this study, Solmon proved in \cite{Solmon} that the $k$-plane transform of $L^p$-functions is well-defined for $1 \leq p < \frac{n}{k}$. Although the $k$-plane transform does not exist for certain radial functions in $L^{\frac{n}{k}} \left( \mathbb{R}^n \right)$, its existence can be obtained for functions in the Lorentz space $L^{\frac{n}{k}, 1} \left( \mathbb{R}^n \right)$. Such a comment can be made by obtaining an ``end-point" boundedness of the $k$-plane transform. For the Radon transform, such an estimate first appeared in \cite{OberlinStein} due to Oberlin and Stein. It was partially extended to the $k$-plane transform for $k \geq \frac{n}{2}$ by Drury in \cite{DruryEndPoint}. Obtaining such end-point estimate for all $1 \leq k \leq n - 1$ is difficult and is still open. However, Duoandikoetxea et al. in \cite{Duoandikoetxea} get the end-point estimate of the $k$-plane transforms of radial functions. Particularly, the authors in \cite{Duoandikoetxea} prove that for radial functions $f$ on $\mathbb{R}^n$, we have
		\begin{equation}
			\label{RnCase}
			\| R_kf \|_{\infty} \leq C \| f \|_{L^{\frac{n}{k}, 1} \left( \mathbb{R}^n \right)},
		\end{equation}
		where $L^{p, 1} \left( \mathbb{R}^n \right)$ is the Lorentz space. With this motivation, Kumar and Ray in \cite{KumarRay} approached the problem of studying the end-point behaviour on certain non-Euclidean spaces. In their paper, they proved these estimates for the hyperbolic space and sphere. We now mention the two results of our interest.
		\begin{theorem}[\cite{KumarRay}]
			For $n \geq 3$, and $2 \leq k \leq n - 1$, there is a constant $C > 0$ such that for all radial functions $f$ on $\mathbb{H}^n$,
			\begin{equation}
				\label{HnCase}
				\| \cosh \left( d \left( e_{n + 1, \cdot} \right) \right) R_kf \left( \cdot \right) \|_{L^{\infty} \left( \Xi_k \right)} \leq C \| f \|_{L^{\frac{n - 1}{k - 1}, 1} \left( \mathbb{H}^n \right)},
			\end{equation}
			where $R_k$ is the $k$-dimensional totally-geodesic Radon transform, and $\Xi_k$ is the collection of all $k$-dimensional totally-geodesic submanifolds of $\mathbb{H}^n$.
		\end{theorem}
		\begin{remark}
			\label{RemarkHnEndPoint}
			\normalfont
			It is to be noted that for $k = 1$, we do not expect any boundedness estimate at the end-point, which is $p = + \infty$. This is because it is known that there are radial functions (say, constants) for which the $k$-plane transform is not defined. Also, in the case $p = + \infty$, there are no non-trivial Lorentz spaces $L^{p, q}$, for $q < + \infty$.
		\end{remark}
		\begin{theorem}[\cite{KumarRay}]
			For $1 \leq k \leq n - 1$, and radial functions $f$ on the sphere $\mathbb{S}^n$, the following inequality holds
			\begin{equation}
				\label{SnCase}
				\| \cos \left( d \left( e_{n + 1, \cdot} \right) \right) R_kf \left( \cdot \right) \|_{L^{\infty} \left( \Xi_k \right)} \leq C \| f \|_{L^{\frac{n}{k}, 1} \left( \mathbb{S}^n \right)},
			\end{equation}
			where $R_k$ is the $k$-dimensional totally-geodesic Radon transform, and $\Xi_k$ is the collection of all $k$-dimensional totally-geodesic submanifolds of $\mathbb{S}^n$.
		\end{theorem}
		Upon a close inspection of Equations \eqref{RnCase}, \eqref{HnCase}, and \eqref{SnCase}, we observe that the coefficients involved with $R_kf$ are closely related to the curvatures of the three spaces under consideration. Therefore, it seems to us that the coefficients must come naturally as we try to get these end-point estimates. However, the proofs of these results, presented in \cite{Duoandikoetxea} and \cite{KumarRay} are independent.
		
		Motivated from our observation about the curvature, we give a proof of the end-point results without getting into the group-theoretic or coordinate approach, but by using the differential geometry of constant curvature spaces. The crux of the proof we present lies in the fact that we can get a unified formula for the $k$-dimensional totally-geodesic Radon transform on constant curvature spaces. This formula shows us that the coefficients involved in Equations \eqref{RnCase}, \eqref{HnCase}, and \eqref{SnCase} are natural to curvature. We also see that once we have the unified formula, the proof of the estimate becomes much simpler.
		
		The article is organized as follows: Section \ref{PreliminariesSection} deals with the preliminaries from differential geometry. In section \ref{MainSection} we give the unified proof of the formula of totally-geodesic $k$-plane transform of radial functions, and prove the end-point estimates. Next, in Section \ref{SpecialFunctionSection}, we give some inequalities for hypergeometric functions as an application to one of our lemma. Lastly, we give some concluding remarks and future scope of work in Section \ref{ConclusionSection}.
	\section{Preliminaries from Differential Geometry}
		\label{PreliminariesSection}
		In this section we brief through important results from Riemannian geometry that we use to get the unified formula for the totally-geodesic $k$-plane transform of radial functions. We assume that the reader is well-equipped with the theory of smooth manifolds, and we begin by introducing Riemannian manifolds. The definitions and results that follow are taken from \cite{LeeRM}, and we present them in a form convenient to us.
		\begin{definition}[Riemannian manifolds]
			Let $M$ be a smooth manifold. A \textbf{Riemannian metric} on $M$ is a section $g$ of the symmetric tensor bundle that is positive definite at each point. A smooth manifold with a prescribed Riemannian metric is known as a \textbf{Riemannian manifold}.
		\end{definition}
		\begin{remark}
			\normalfont
			In simple terms, a Riemannian manifold is a smooth manifold such that on each tangent space we have prescribed an inner product. We require this prescription to vary smoothly.
		\end{remark}
		With the introduction of a Riemannian metric on a smooth manifold, various constructions can be done. Some of them include geodesics, distance function, exponential maps, volume forms and measures. We refer the reader to \cite{LeeRM} for an extensive treatise on this topic. Of our interest, however, are those submanifolds which behave like planes of the Euclidean space. Based on the observation that the lines of lower-dimensional planes in the Euclidean space are also lines of the universe ($\mathbb{R}^n$), we look at those submanifolds whose geodesics are geodesics of the parent manifold. Formally, these submanifolds are called totally-geodesic and are defined as follows.
		\begin{definition}[Totally-Geodesic Submanifolds]
			Let $\left( M, g \right)$ be a Riemannian manifold and $S \subseteq M$ be an embedded submanifold of $M$. Then, $S$ is totally-geodesic if the geodesics in $S$ (with respect to the induced metric) are also geodesics of $M$.
		\end{definition}
		\begin{remark}
			\normalfont
			We remark that singleton sets of a Riemannian manifold are embedded submanifolds. However, being of dimension zero, they have the trivial tangent space, and no non-trivial geodesics to work with. Therefore, vacously, singleton sets become totally-geodesic submanifolds.
		\end{remark}
		On a Riemannian manifold, one can define various curvature tensors (see \cite{LeeRM} for details). In this article we deal with those manifolds which have constant \textbf{sectional curvature}. The classification of such spaces is well-known, and can be found in \cite{LeeRM}. We state the result here for reference.
		\begin{theorem}[Classification of simply-connected spaces of constant curvature]
			\label{Classification}
			Let $\left( M, g \right)$ be a simply-connected Riemannian manifold of constant sectional curvature $c$. Then, $M$ is isometric to one of the spaces: $\mathbb{R}^n$, $\mathbb{H}^n_R$, and $\mathbb{S}^n_R$, depending on if $c = 0$, $c = - \frac{1}{R^2} < 0$, or $c = \frac{1}{R^2} > 0$, respectively.
		\end{theorem}
		\begin{remark}
			\normalfont
			Theorem \ref{Classification} is of fundamental importance in Differential Geometry, and to us. It says that apart than the three ``model spaces", there are no other (simply-connected) manifolds of constant curvature. Due to this result, it is enough to study the $k$-plane transform only on the model spaces $\mathbb{R}^n$, $\mathbb{H}^n_R$, and $\mathbb{S}^n_R$. In fact, since $\mathbb{H}^n_R$ is diffeomorphic to $\mathbb{H}^n$, and $\mathbb{S}^n_R$ is diffeomorphic to $\mathbb{S}^n$ through an orientation preserving diffeomorphism (given by simply a rescaling of the space), the measures on these spaces are only a scaling of each other. Hence, it is enough to study the totally-geodesic Radon transforms on $\mathbb{R}^n$, $\mathbb{H}^n$, and $\mathbb{S}^n$, which have curvatures $c = 0, -1$, and $+1$, respectively.
			\end{remark}
			One of the important tools on which the proof we present depends is the polar decomposition of the measure on a manifold of constant curvature. For this, we define the following function, that we call ``curvature function". This function comes naturally when studying Jacobi fields on constant curvature spaces (see \cite{LeeRM}), as solution to the differential equation
			\begin{equation}
				\label{CurvatureDE}
				u'' \left( t \right) + c u \left( t \right) = 0,
			\end{equation}
			with the initial conditions $u \left( 0 \right) = 0$, and $u' \left( 0 \right) = 1$. Explicitly, the solution of Equation \eqref{CurvatureDE} with the aforementioned initial conditions is given by
			\begin{equation}
				\label{CurvatureFunction}
				s_c \left( t \right) := \begin{cases}
											t, & \text{if } c = 0. \\
											R \sinh \frac{t}{R}, & \text{if } c = -\frac{1}{R^2} < 0. \\
											R \sin \frac{t}{R}, & \text{if } c = +\frac{1}{R^2} > 0.
										\end{cases}
			\end{equation}
			Henceforth, we are interested only in the case when $c \in \left\lbrace -1, 0, 1 \right\rbrace$. That is, we consider
			\begin{equation}
				\label{OurCurvatureFunction}
				s_c \left( t \right) = \begin{cases}
											t, & \text{if } c = 0. \\
											\sinh t, & \text{if } c = -1. \\
											\sin t, & \text{if } c = +1.
										\end{cases}
			\end{equation}
			The main aim of this article is to unify the formula for the $k$-plane transform of radial functions on constant curvature spaces. This involves various properties of the function $s_c$ mentioned in Equation \eqref{OurCurvatureFunction}. While those properties are easily derivable by fixing a $c$, and using the knowledge of trigonometry, we give certain proofs that do not invoke trigonometry but rather only Equation \eqref{CurvatureDE}.
			
			First, we observe that the function $s_c$ must be odd. To see this, consider the function $v \left( t \right) = s_c \left( - t \right)$. Clearly, $v$ is a solution to Equation \eqref{CurvatureDE} with the initial conditions $v \left( 0 \right) = 0$ and $v' \left( 0 \right) = -1$. From the uniqueness of solutions to differential equations, we have $v \left( t \right) = K s_c \left( t \right)$, for some $K \in \mathbb{R}$. It is clear from the initial conditions that $K = -1$.
			
			We know that the function $s_c$, given in Equation \eqref{OurCurvatureFunction}, satisfies Equation \eqref{CurvatureDE}. By multiplying $2 s_c' \left( t \right)$ on both sides, it is easy to see that
			$$\left( \left( s_c' \right)^2 \left( t \right) + cs_c^2 \left( t \right) \right)' = 0.$$
			Consequently, we must have for any $t \in \mathbb{R}$,
			$$\left( s_c' \right)^2 \left( t \right) + cs_c^2 \left( t \right) = K,$$
			for some constant $K \in \mathbb{R}$. From the initial conditions $s_c \left( 0 \right) = 0$ and $s_c' \left( 0 \right) = 1$, it is clear that $K = 1$. That is, the function $s_c$ satisfies the following non-linear differential equation.
			\begin{equation}
				\label{ScNonLinDE}
				\left( s_c' \right)^2 + cs_c^2 = 1.
			\end{equation}
			A general solution to Equation \eqref{CurvatureDE} is given by
			$$u \left( t \right) = A s_c \left( t \right) + B s_c' \left( t \right),$$
			where $A$ and $B$ are uniquely determined by the initial conditions. Now, by taking the initial conditions as $u \left( 0 \right) = s_c \left( r \right)$ and $u' \left( 0 \right) = s_c' \left( r \right)$, for a fixed $r \in \mathbb{R}$, we get
			$$u \left( t \right) = s_c \left( t \right) s_c' \left( r \right) + s_c \left( r \right) s_c' \left( t \right).$$
			However, by the uniqueness of solutions to initial value problems, we must have $u \left( t \right) = s_c \left( t + r \right)$. That is, we have the following ``\textit{double-angle formula}".
			\begin{equation}
				\label{DoubleAngleFormula2}
				s_c \left( t + r \right) = s_c \left( t \right) s_c' \left( r \right) + s_c \left( r \right) s_c' \left( t \right).
			\end{equation}
			Differentiating Equation \eqref{DoubleAngleFormula2}, and using Equation \eqref{CurvatureDE}, we obtain
			\begin{equation}
				\label{DoubleAngleFormula5}
				s_c' \left( t + r \right) = s_c' \left( t \right) s_c' \left( r \right) - c s_c \left( t \right) s_c \left( r \right). 
			\end{equation}
			As a special case of Equation \eqref{DoubleAngleFormula2}, by taking $r = t$, we get
			\begin{equation}
				\label{DoubleAngleFormula1}
				s_c \left( 2t \right) = 2 s_c \left( t \right) s_c' \left( t \right).
			\end{equation}
			Similarly, by taking $r = t$ in Equation \eqref{DoubleAngleFormula5} and using Equation \eqref{ScNonLinDE}, we obtain
			\begin{equation}
				\label{DoubleAngleFormula3}
				s_c' \left( 2t \right) = 1 - 2cs_c^2 \left( t \right).
			\end{equation}
			Lastly, we consider the following. By using Equations \eqref{DoubleAngleFormula2} and \eqref{DoubleAngleFormula5}, we have
			\begin{align*}
				2s_c \left( \frac{t + r}{2} \right) s_c' \left( \frac{t - r}{2} \right) &= 2 \left[ s_c \left( \frac{t}{2} \right) s_c' \left( \frac{r}{2} \right) + s_c \left( \frac{r}{2} \right) s_c' \left( \frac{t}{2} \right) \right] \left[ s_c' \left( \frac{t}{2} \right) s_c' \left( \frac{r}{2} \right) + cs_c \left( \frac{t}{2} \right) s_c \left( \frac{r}{2} \right) \right] \\
				&= 2 \left[ s_c \left( \frac{t}{2} \right) s_c' \left( \frac{t}{2} \right) \left( s_c' \right)^2 \left( \frac{r}{2} \right) + cs_c^2 \left( \frac{t}{2} \right) s_c \left( \frac{r}{2} \right) s_c' \left( \frac{r}{2} \right) \right. \\
				&\left. + s_c \left( \frac{r}{2} \right) s_c' \left( \frac{r}{2} \right) \left( s_c' \right)^2 \left( \frac{t}{2} \right) + cs_c^2 \left( \frac{r}{2} \right) s_c \left( \frac{t}{2} \right) s_c' \left( \frac{t}{2} \right) \right].
			\end{align*}
			Using Equation \eqref{DoubleAngleFormula1} and simplifying, we get
			\begin{align*}
				2s_c \left( \frac{t + r}{2} \right) s_c' \left( \frac{t - r}{2} \right) &= s_c \left( t \right) \left[ \left( s_c' \right)^2 \left( \frac{r}{2} \right) + cs_c^2 \left( \frac{r}{2} \right) \right] + s_c \left( r \right) \left[ \left( s_c' \right)^2 \left( \frac{t}{2} \right) + c s_c^2 \left( \frac{t}{2} \right) \right].
			\end{align*}
			Finally, from Equation \eqref{ScNonLinDE}, we have,
			\begin{equation}
				\label{DoubleAngleFormula4}
				s_c \left( t \right) + s_c \left( r \right) = 2 s_c \left( \frac{t + r}{2} \right) s_c \left( \frac{t - r}{2} \right).
			\end{equation}
			One significant application of the function $s_c$ is the decomposition of the measure on constant curvature spaces into a ``radial" part and an ``angular" part. The next result that can be found in \cite{LeeRM} is the precise statement for polar decomposition.
			\begin{theorem}[Polar Decomposition \cite{LeeRM}]
				Let $\left( M, g \right)$ be a Riemannian manifold of constant curvature $c$, and let $U \subseteq M$ be a geodesic ball centered at a point $p \in M$ and of radius $b > 0$. Then, for any integrable function $f$ on $U$, we have
				\begin{equation}
					\label{PD}
					\int\limits_{U} f \ \mathrm{d}V_g = \int\limits_{\mathbb{S}^{n - 1}} \int\limits_{0}^{b} f \left( \exp_p \left( r \omega \right) \right) s_c^{n - 1} \left( r \right) \mathrm{d}r \mathrm{d}\omega.
				\end{equation}
				Here, $r$ denotes the radial distance of points in the geodesic ball $U$ from $p$, and $\omega$ is their direction.
			\end{theorem}
			\begin{remark}
				\normalfont
				One of the advantages of the model Riemannian spaces ($\mathbb{R}^n$, $\mathbb{H}^n$, and $\mathbb{S}^n$) is that they are geodesic balls with respect to any point as the center. Therefore, on these spaces, we can polar decompose any integral keeping any (fixed) point as the ``center" of the manifold.
			\end{remark}
			In this article, we talk about spaces of constant curvature. To have a better understanding, one often visualizes these spaces through some models. The Euclidean space and the sphere, indeed, have natural models. However, the hyperbolic space can be described through different models. To keep a geometrical insight in the background, we choose the upper sheet of hyperboloid model for $\mathbb{H}^n$. For other models available in literature, we refer the reader to \cite{LeeRM}. Also, to make things a bit simpler, we consider $\mathbb{R}^n$ with an embedding inside $\mathbb{R}^{n + 1}$ as $\mathbb{R}^n \simeq \left\lbrace x \in \mathbb{R}^{n + 1} | x_{n + 1} = 1 \right\rbrace$. With this embedding, all three spaces of constant curvature can be visualized simultaneously as in Figure \ref{SimultaneousFigure}. The point $e_{n + 1} \in \mathbb{R}^{n + 1}$ can also serve as a ``common" origin to all three spaces, and we denote it by $o$.
			\begin{figure}[ht!]
				\centering
				\includegraphics[scale=0.17]{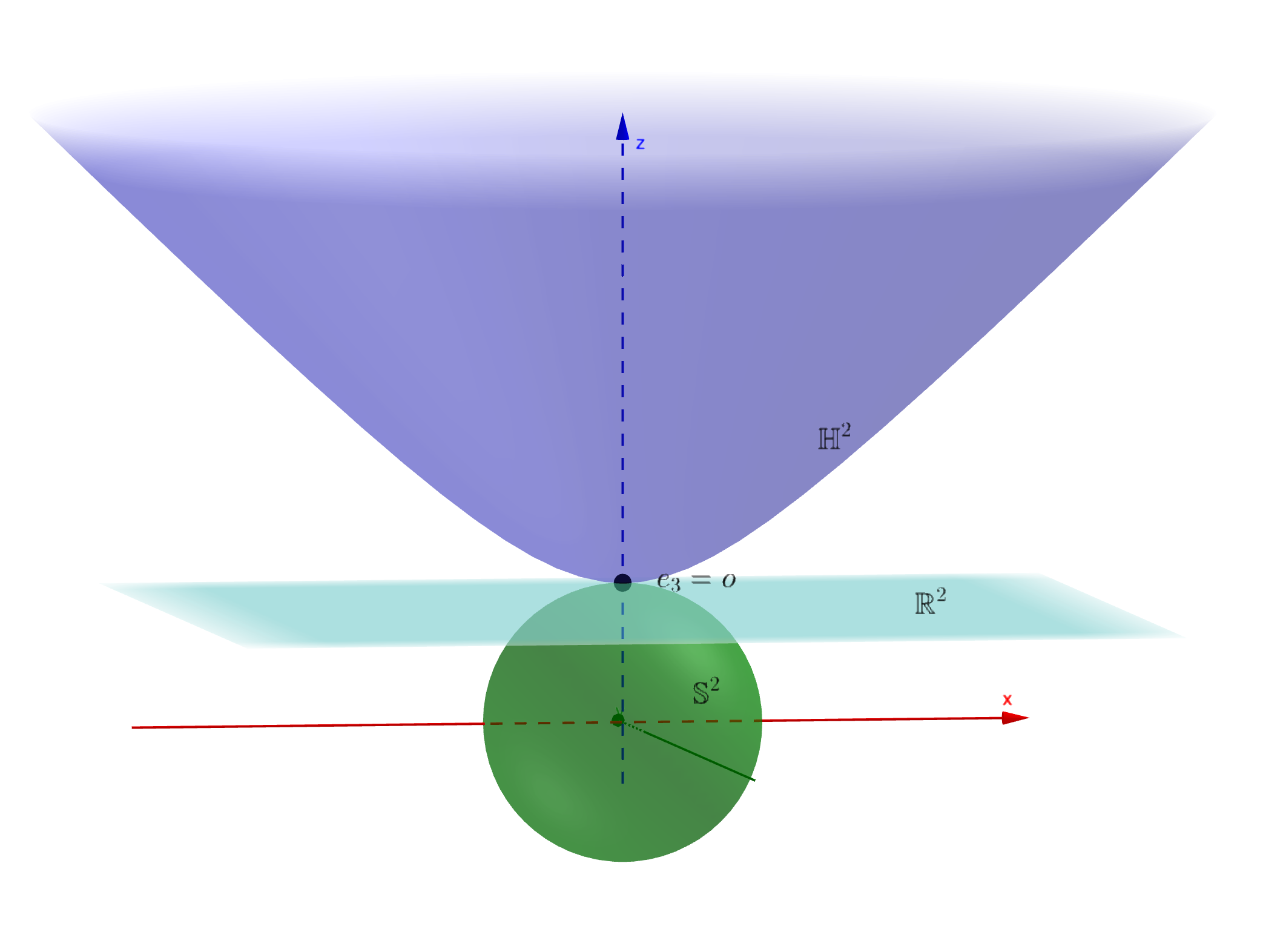}
				\caption{The three model spaces embedded in a higher dimensional Euclidean space.}
				\label{SimultaneousFigure}
			\end{figure}
			
			The following characterization of $k$-dimensional totally-geodesic submanifolds in the model Riemannian spaces is of importance to us. This can be found, though independently, in \cite{LeeRM}, \cite{HelgasonGGA}, or \cite{BerensteinRubin}.
		\begin{theorem}
			\label{CharacterizationTGS}
			Let $X$ be a simply-connected space of constant curvature $c$. The $k$-dimensional totally-geodesic submanifolds of $X$ are the intersections of $X$ with $\left( k + 1 \right)$-dimensional subspaces of $\mathbb{R}^{n + 1}$.
		\end{theorem}
		In simple terms we say that the totally-geodesic submanifolds of the model Riemannian spaces are lower dimensional versions of themselves. An important consequence of Theorem \ref{CharacterizationTGS} is that the totally-geodesic submanifolds can be viewed as geodesic balls with any point (in them) as the center. Hence, we can polar decompose integrals on totally-geodesic submanifolds by fixing a point.
		
		We now move on to give the precise definition of the totally-geodesic $k$-plane transform.
		\begin{definition}[$k$-plane transform]
			\label{KPlaneTransformX}
			Let $X$ be a simply connected Riemannian manifold of constant curvature $c$, and let $\Xi_k \left( X \right)$ be the collection of all $k$-dimensional totally-geodesic submanifolds in $X$. Given a ``nice" function $f: X \rightarrow \mathbb{C}$, its $k$-plane transform is a function $R_kf: \Xi_k \left( X \right) \rightarrow \mathbb{C}$ defined as
			\begin{equation}
				\label{KPlaneTransformEquation}
				R_kf \left( \xi \right) = \int\limits_{\xi} f \left( x \right) \mathrm{d}_{\xi}x,
			\end{equation}
			where, $d_{\xi}x$ is the measure induced from the (induced) Riemannian metric of $\xi$.
		\end{definition}
		We use the following result in Section \ref{MainSection} to get a unified formula for the totally-geodesic $k$-plane transform of radial functions.
		\begin{theorem}[\cite{LeeRM}]
			\label{PerpendicularGeodesics}
			Let $\left( M, g \right)$ be a Riemannian manifold, $p \in M$, and $S \subseteq M$ be a (closed) totally-geodesic submanifold such that $p \notin S$. Then, there is a point $q \in S$ such that $d \left( p, S \right) = d \left( p, q \right)$. Moreover, the geodesic segment joining $p$ and $q$ intersects $S$ orthogonally.
		\end{theorem}
		Lastly, we state a unified Pythagoras' theorem in the three model spaces. We see that the Pythagoras' theorem plays an important role in deriving a formula for the $k$-plane transform of radial functions. With this link unified, we expect that the formula for the $k$-plane transform, and the results concerning the end-point estimates can also be unified. While there are various ways to look at Pythagoras' theorem, the one we state can be found in \cite{UnifiedPTFoote}. The author of \cite{UnifiedPTFoote} emphasises the importance of looking at the result in terms of area, rather than simply the length of the sides (as in \cite{Ratcliffe}).
		\begin{theorem}[\cite{UnifiedPTFoote}]
			Let $X$ be a space of constant curvature $c$ and let $p, q, r \in X$. Suppose that the geodesic triangle formed by these points is a right-angled triangle with the right-angle at $q$. Suppose the length of the side $pq$ is $a$, that of $qr$ is $b$ and that of $pr$ be $h$. Then, we have,
			\begin{equation}
				\label{UnifiedPT}
				A \left( h \right) = A \left( a \right) + A \left( b \right) - \frac{c}{2 \pi} A \left( a \right) A \left( b \right),
			\end{equation}
			where, $A \left( r \right)$ is the area of a disc in $X$ with radius $r$.
		\end{theorem}
		It is known (see for instance, \cite{UnifiedPTFoote}) that the area of a disc of radius $r$ in a space $X$ of constant curvature $c$ is given by
		\begin{equation}
			\label{AreaX}
			A \left( r \right) = 4 \pi s_c^2 \left( \frac{r}{2} \right).
		\end{equation}
		Using Equation \eqref{AreaX} into Equation \eqref{UnifiedPT}, and simplifying, we get the following version of the unified Pythagoras' theorem.
		\begin{equation}
			\label{UnifiedPT2}
			s_c^2 \left( \frac{h}{2} \right) = s_c^2 \left( \frac{a}{2} \right) + s_c^2 \left( \frac{b}{2} \right) - 2c s_c^2 \left( \frac{a}{2} \right) s_c^2 \left( \frac{b}{2} \right).
		\end{equation}
		We make use of Equation \eqref{UnifiedPT2} throughout our calculations. With this background, we now proceed to give the main results of this article.
	\section{End-Point Estimates on constant curvature spaces}
		\label{MainSection}
		The main result of this section depends on a unified formula for the totally-geodesic $k$-plane transform of radial functions for spaces of constant curvature. The models of these spaces are fixed as in Section \ref{PreliminariesSection}, with their ``origin" denoted by $o$. We wish to mention here that while considering the $k$-plane transform on the sphere, it is enough to consider even functions, since odd functions are in the null space of the $k$-plane transform (see \cite{KumarRay}). Consequently, it is enough to consider only the half-sphere. Keeping this in mind, we make the following definition of radial functions.
		\begin{definition}[Radial functions]
			\label{RadialFunctions}
			Let $X$ be a space of constant curvature $c$. A function $f: X \rightarrow \mathbb{C}$ is radial if there is a function $\tilde{f}: s_c \left( \left[ 0, \frac{\delta \left( X \right)}{4} \right) \right) \rightarrow \mathbb{C}$ such that $f \left( x \right) = \tilde{f} \left( s_c \left( \frac{d \left( o, x \right)}{2} \right) \right)$.
		\end{definition}
		\begin{remark}
			\normalfont
			The definition we use here seems to be different from that found in literature. Particularly, authors often use $\cos$ and $\cosh$ for radial functions on the sphere and the hyperbolic space, respectively (see, for instance, \cite{KumarRay}). We see that Definition \ref{RadialFunctions} is convenient to our needs. Also, we would like to remark that while the definition seems to be different, it is equivalent to the usual definitions.
		\end{remark}
		We are now in a position to give the unified formula for the totally-geodesic $k$-plane transform of radial functions.
		\begin{theorem}
			\label{TGRTRadialFunction}
			Let $X$ be a simply connected space of constant curvature $c$, and $f: X \rightarrow \mathbb{C}$ be a radial function. Then,
			\begin{equation}
				\label{RadialFunctionEquation}
				s_c' \left( d \left( o, \xi \right) \right) R_kf \left( \xi \right) = K \int\limits_{d \left( o, \xi \right)}^{\frac{\delta \left( X \right)}{2}} \tilde{f} \left( s_c \left( \frac{t}{2} \right) \right) \left[ 1 - \left( \dfrac{\left( \ln s_c \right)' \left( t \right)}{\left( \ln s_c \right)' \left( d \left( o, \xi \right) \right)} \right)^2 \right]^{\frac{k}{2} - 1} s_c^{k - 1} \left( t \right) \mathrm{d}t,
			\end{equation}
			where $K$ is a constant depending only on the dimension of the totally-geodesic submanifolds under consideration, and the curvature $c$. Particularly, we have
			$$K = \begin{cases}
						\left| \mathbb{S}^{k - 1} \right|, & c \neq +1. \\
						2 \left| \mathbb{S}^{k - 1} \right|, & c = +1.
					\end{cases}$$
		\end{theorem}
		\begin{proof}
			Since $\xi$ is a totally-geodesic submanifold of $X$, we conclude from Theorem \ref{PerpendicularGeodesics} that there is some $x_0 \in \xi$ such that $d \left( o, \xi \right) = d  \left( o, x_0 \right)$. Moreover, the geodesic triangle formed by $o, x_0$ and any $x \in \xi$ is a right-angled triangle with the right-angle at $x_0$. We consider the polar decomposition of $\xi$ with respect to $x_0$ and use $r = d \left( x, x_0 \right)$, in this proof. We then have,
			\begin{equation}
				\label{F1}
				R_kf \left( \xi \right) = | \mathbb{S}^{k - 1} | \int\limits_{0}^{\delta \left( X \right)} \tilde{f} \left( s_c \left( \frac{d \left( o, x \right)}{2} \right) \right) s_c^{k - 1} \left( r \right) \mathrm{d}r.
			\end{equation}
			Two points are of interest here: One, since the triangle formed by $o, x$, and $x_0$ is right-angled at $x_0$ (see Figure \ref{RightTriangleFigure}), we have that $s_c \left( \frac{d \left( o, x \right)}{2} \right)$ is a function of $r$, owing to the Pythagoras' theorem. Second, if $X$ is the sphere and $f$ is an even function, then the integral can be reduced from $\left( 0, \delta \left( X \right) \right)$ to $\left( 0, \frac{\delta \left( X \right)}{2} \right)$. However, this brings in a factor of $2$, which is not present in the other two spaces. Therefore, henceforth, we will not be worried about the exact value of the constant, but only mention it by $K$. The geodesic triangles under consideration in different spaces are shown in Figure \ref{RightTriangleFigure}.
			\begin{figure}[ht!]
				\centering
				\begin{subfigure}[t]{1.0\textwidth}
					\centering
					\includegraphics[scale=0.275]{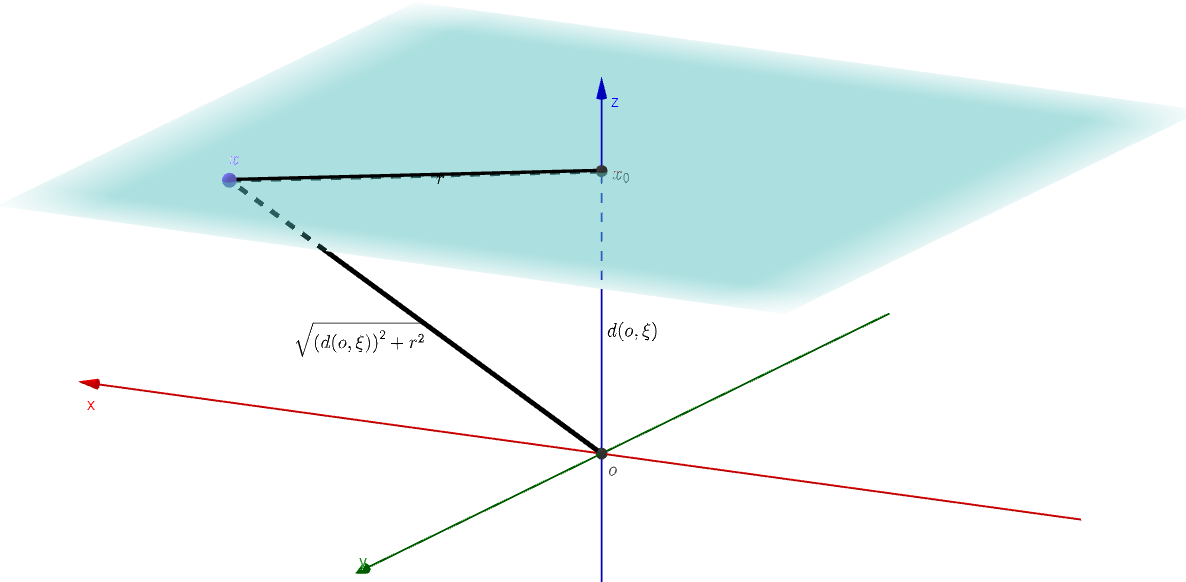}
					\caption{Distance of $x$ from $o$ in terms of $d \left( o, \xi \right)$ and $r$ in $\mathbb{R}^n$.}
					\label{RightTriangleKTGSRn}
				\end{subfigure}
				\\
				\begin{subfigure}[t]{1.0\textwidth}
					\centering
					\includegraphics[scale=0.275]{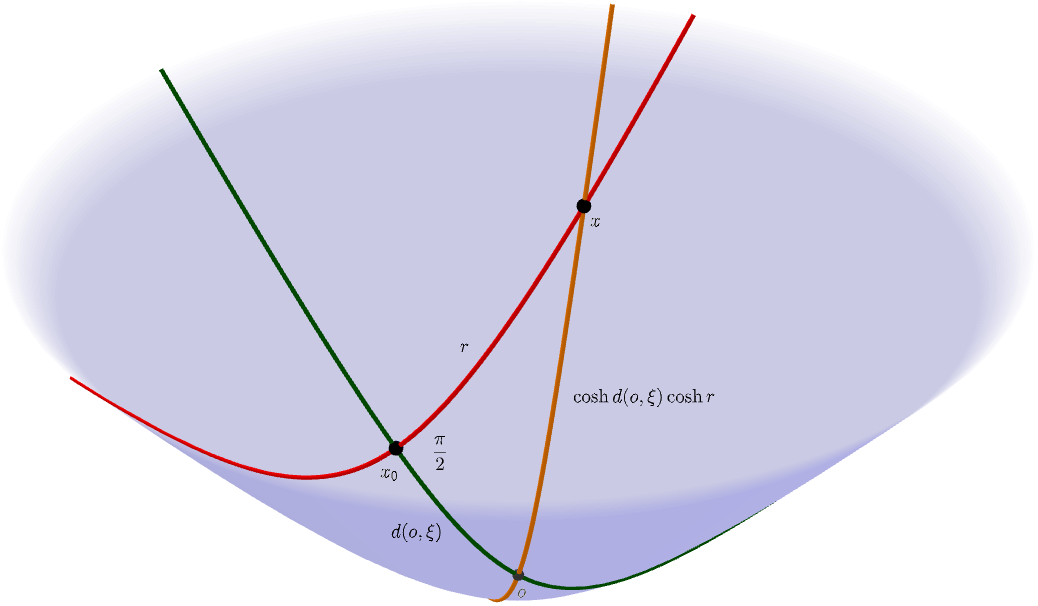}
					\caption{Distance of $x$ from $o$ in terms of $d \left( o, \xi \right)$ and $r$ in $\mathbb{H}^n$.}
					\label{RightTriangleKTGSHn}
				\end{subfigure}
				\\
				\vspace{0.5cm}
				\begin{subfigure}[t]{1.0\textwidth}
					\centering
					\includegraphics[scale=0.275]{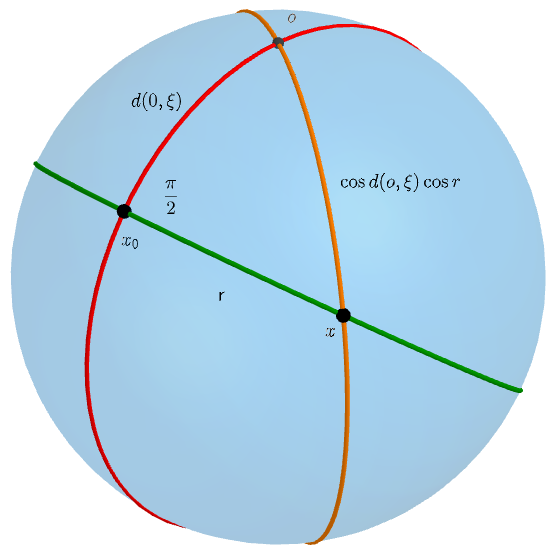}
					\caption{Distance of $x$ from $o$ in terms of $d \left( o, \xi \right)$ and $r$ in $\mathbb{S}^n$.}
					\label{RightTriangleKTGSSn}
				\end{subfigure}
				\caption{The right-angled triangles in spaces of constant curvature}
				\label{RightTriangleFigure}
			\end{figure}
			
			Using Equation \eqref{UnifiedPT2} (for the triangle formed by $0$, $x_0$ and $x$), we have,
			\begin{equation}
				\label{PTSc}
				s_c^2 \left( \frac{d \left( o, x \right)}{2} \right) = s_c^2 \left( \frac{d \left( o, \xi \right)}{2} \right) + s_c^2 \left( \frac{r}{2} \right) \left[ 1 - 2cs_c^2 \left( \frac{d \left( o, \xi \right)}{2} \right) \right].
			\end{equation}
			In Equation \eqref{F1}, let us make the change of variable $s_c^2 \left( \frac{d \left( o, x \right)}{2} \right) = s_c^2 \left( \frac{t}{2} \right)$. From Equations \eqref{PTSc}, \eqref{DoubleAngleFormula1}, and \eqref{DoubleAngleFormula3}, we have the following change of measure
			\begin{equation}
				\label{MeasureChange}
				s_c \left( r \right) \mathrm{d}r = \frac{s_c \left( t \right)}{s_c' \left( d \left( o, \xi \right) \right)} \ \mathrm{d}t.
			\end{equation}
			From our substitution and Equation \eqref{DoubleAngleFormula3}, we have
			\begin{equation}
				\label{SCR2}
				s_c^2 \left( \frac{r}{2} \right) = \frac{s_c^2 \left( \frac{t}{2} \right) - s_c^2 \left( \frac{d \left( o, \xi \right)}{2} \right)}{s_c' \left( d \left( o, \xi \right) \right)}.
			\end{equation}
			Therefore, using Equation \eqref{ScNonLinDE}, we get
			\begin{align*}
				\left( s_c' \right)^2 \left( \frac{r}{2} \right) = 1 - cs_c^2 \left( \frac{r}{2} \right) &= 1 - \frac{c s_c^2 \left( \frac{t}{2} \right) - cs_c^2 \left( \frac{d \left( o, \xi \right)}{2} \right)}{1 - 2cs_c^2 \left( \frac{d \left( o, \xi \right)}{2} \right)}
				= \frac{1 - cs_c^2 \left( \frac{t}{2} \right) - cs_c^2 \left( \frac{d \left( o, \xi \right)}{2} \right)}{1 - 2cs_c^2 \left( \frac{d \left( o, \xi \right)}{2} \right)}.
			\end{align*}
			Again, by using Equations \eqref{ScNonLinDE} and \eqref{DoubleAngleFormula3}, we get
			$$\left( s_c' \right)^2 \left( \frac{r}{2} \right) = \frac{ \left( s_c' \right)^2 \left( \frac{d \left( o, \xi \right)}{2} \right) + \left( s_c' \right)^2 \left( \frac{t}{2} \right) - 1}{s_c' \left( d \left( o, \xi \right) \right)}.$$
			Therefore, from Equation \eqref{DoubleAngleFormula1}, we get
			\begin{align}
				s_c^2 \left( r \right) &= 4s_c^2 \left( \frac{r}{2} \right) \left( s_c' \right)^2 \left( \frac{r}{2} \right) \nonumber \\
				&= 4 \left( \frac{s_c^2 \left( \frac{t}{2} \right) - s_c^2 \left( \frac{d \left( o, \xi \right)}{2} \right)}{s_c' \left( d \left( o, \xi \right) \right)} \right) \left( \frac{ \left( s_c' \right)^2 \left( \frac{d \left( o, \xi \right)}{2} \right) + \left( s_c' \right)^2 \left( \frac{t}{2} \right) - 1}{s_c' \left( d \left( o, \xi \right) \right)} \right) \nonumber \\
				&= \frac{4 s_c^2 \left( \frac{d \left( o, \xi \right)}{2} \right) \left( s_c' \right)^2 \left( \frac{d \left( o, \xi \right)}{2} \right)}{\left( s_c' \right)^2 \left( d \left( o, \xi \right) \right)} \left( \frac{s_c^2 \left( \frac{t}{2} \right)}{s_c^2 \left( \frac{d \left( o, \xi \right)}{2} \right)} - 1 \right) \left( 1 + \frac{\left( s_c' \right)^2 \left( \frac{t}{2} \right) - 1}{\left( s_c' \right)^2 \left( \frac{d \left( o, \xi \right)}{2} \right)} \right) \nonumber \\
				\label{Sc2r}
				&= \left( \frac{s_c \left( d \left( o, \xi \right) \right)}{s_c' \left( d \left( o, \xi \right) \right)} \right)^2 \left[ \frac{s_c^2 \left( \frac{t}{2} \right)}{s_c^2 \left( \frac{d \left( o, \xi \right)}{2} \right)} + \frac{s_c^2 \left( \frac{t}{2} \right) \left( \left( s_c' \right)^2 \left( \frac{t}{2} \right) - 1 \right)}{s_c^2 \left( \frac{d \left( o, \xi \right)}{2} \right) \left( s_c' \right)^2 \left( \frac{d \left( o, \xi \right)}{2} \right)} - 1 - \frac{\left( s_c' \right)^2 \left( \frac{t}{2} \right) - 1}{\left( s_c' \right)^2 \left( \frac{d \left( o, \xi \right)}{2} \right)} \right]. \nonumber
			\end{align}
			Now, we have,
			\begin{align*}
				&\frac{s_c^2 \left( \frac{t}{2} \right)}{s_c^2 \left( \frac{d \left( o, \xi \right)}{2} \right)} + \frac{s_c^2 \left( \frac{t}{2} \right) \left( \left( s_c' \right)^2 \left( \frac{t}{2} \right) - 1 \right)}{s_c^2 \left( \frac{d \left( o, \xi \right)}{2} \right) \left( s_c' \right)^2 \left( \frac{d \left( o, \xi \right)}{2} \right)} - 1 - \frac{\left( s_c' \right)^2 \left( \frac{t}{2} \right) - 1}{\left( s_c' \right)^2 \left( \frac{d \left( o, \xi \right)}{2} \right)} \\
				&= \frac{1}{s_c^2 \left( \frac{d \left( o, \xi \right)}{2} \right) \left( s_c' \right)^2 \left( \frac{d \left( o, \xi \right)}{2} \right)} \left[ s_c^2 \left( \frac{t}{2} \right) \left( s_c' \right)^2 \left( \frac{d \left( o, \xi \right)}{2} \right) + s_c^2 \left( \frac{t}{2} \right) \left( s_c' \right)^2 \left( \frac{t}{2} \right) - s_c^2 \left( \frac{t}{2} \right) \right. \\
				&\left. - \left( s_c' \right)^2 \left( \frac{d \left( o, \xi \right)}{2} \right) s_c^2 \left( \frac{d \left( o, \xi \right)}{2} \right) - \left( s_c' \right)^2 \left( \frac{t}{2} \right) s_c^2 \left( \frac{d \left( o, \xi \right)}{2} \right) + s_c^2 \left( \frac{d \left( o, \xi \right)}{2} \right) \right].
			\end{align*}
			Using Equations \eqref{DoubleAngleFormula1}, \eqref{DoubleAngleFormula2}, and \eqref{DoubleAngleFormula4}, we get
			\begin{align}
				&\frac{s_c^2 \left( \frac{t}{2} \right)}{s_c^2 \left( \frac{d \left( o, \xi \right)}{2} \right)} + \frac{s_c^2 \left( \frac{t}{2} \right) \left( \left( s_c' \right)^2 \left( \frac{t}{2} \right) - 1 \right)}{s_c^2 \left( \frac{d \left( o, \xi \right)}{2} \right) \left( s_c' \right)^2 \left( \frac{d \left( o, \xi \right)}{2} \right)} - 1 - \frac{\left( s_c' \right)^2 \left( \frac{t}{2} \right) - 1}{\left( s_c' \right)^2 \left( \frac{d \left( o, \xi \right)}{2} \right)} \nonumber \\
				&= \frac{1}{\frac{1}{4} s_c^2 \left( d \left( o, \xi \right) \right)} \left[ s_c \left( \frac{t + d \left( o, \xi \right)}{2} \right) s_c \left( \frac{t - d \left( o, \xi \right)}{2} \right) \right. \nonumber \\
				&\left. + s_c \left( \frac{t + d \left( o, \xi \right)}{2} \right) s_c' \left( \frac{t - d \left( o, \xi \right)}{2} \right) s_c \left( \frac{t - d \left( o, \xi \right)}{2} \right) s_c' \left( \frac{t + d \left( o, \xi \right)}{2} \right) \right. \\
				&\left. - s_c \left( \frac{t + d \left( o, \xi \right)}{2} \right) s_c \left( \frac{t - d \left( o, \xi \right)}{2} \right) \right] \nonumber \\
				&= \frac{s_c \left( t + d \left( o, \xi \right) \right) s_c \left( t - d \left( o, \xi \right) \right)}{s_c^2 \left( d \left( o, \xi \right) \right)} \nonumber \\
				&= \frac{s_c^2 \left( t \right) \left( s_c' \right)^2 \left( d \left( o, \xi \right) \right) - \left( s_c' \right)^2 \left( t \right) s_c^2 \left( d \left( o, \xi \right) \right)}{s_c^2 \left( d \left( o, \xi \right) \right)} \nonumber \\
				\label{BracketedTermSimplified}
				&= \left( \frac{s_c' \left( d \left( o, \xi \right) \right)}{s_c \left( d \left( o, \xi \right) \right)} \right)^2 s_c^2 \left( t \right) \left[ 1 - \frac{s_c^2 \left( d \left( o, \xi \right) \right) \left( s_c' \right)^2 \left( t \right)}{\left( s_c' \right)^2 \left( d \left( o, \xi \right) \right) s_c^2 \left( t \right)} \right].
			\end{align}
			Substituting Equation \eqref{BracketedTermSimplified} in Equation \eqref{Sc2r}, we get
			\begin{equation}
				\label{SCR}
				s_c \left( r \right) = s_c \left( t \right) \left[ 1 - \frac{s_c^2 \left( d \left( o, \xi \right) \right) \left( s_c' \right)^2 \left( t \right)}{\left( s_c' \right)^2 \left( d \left( o, \xi \right) \right) s_c^2 \left( t \right)} \right]^{\frac{1}{2}}.
			\end{equation}
			Using Equations \eqref{MeasureChange} and \eqref{SCR} in Equation \eqref{F1}, we get
			\begin{align*}
				R_kf \left( \xi \right) &= \frac{K}{s_c' \left( d \left( o, \xi \right) \right)} \int\limits_{d \left( o, \xi \right)}^{\frac{\delta \left( X \right)}{2}} \tilde{f} \left( s_c \left( \frac{t}{2} \right) \right) \left[ 1 - \frac{s_c^2 \left( d \left( o, \xi \right) \right) \left( s_c' \right)^2 \left( t \right)}{\left( s_c' \right)^2 \left( d \left( o, \xi \right) \right) s_c^2 \left( t \right)} \right]^{\frac{k}{2} - 1} s_c^{k - 1} \left( t \right) \mathrm{d}t.
			\end{align*}
			This is same as Equation \eqref{RadialFunctionEquation}.
		\end{proof}
		Now that we have a unified formula for the totally-geodesic $k$-plane transform for radial functions, we can give a proof of the end-point estimates for all three spaces together. For this, we will employ the following lemma. The result of this lemma is motivated from \cite{KumarRay}.
		\begin{lemma}
			\label{MainLemmaUnified}
			Let $\eta_1, \eta_2 > 0$ and $f$ be a characteristic function of a measurable subset of $\left( 0, \infty \right)$ with finite measure. For a given quantity $A$, let $A_+ := \max \left\lbrace A, 0 \right\rbrace$. Also, let $H$ be the heavyside function defined as $H \left( s \right) = \begin{cases}
				1, & s \geq 0. \\
				0, & s < 0.
			\end{cases}$. Then, for $c \in \left\lbrace -1, 0, 1 \right\rbrace$ and for every $p \geq 1$, we have,
			\begin{equation}
				\label{MainLemmaEquationUnified}
				\begin{aligned}
					&\int\limits_{0}^{\infty} f \left( t \right) t^{\frac{\eta_1 - H \left( c \right) - 1}{2}} \left( 1 - ct \right)_+^{\frac{\eta_2 - H \left( c \right) - 1}{2}} \mathrm{d}t \leq K \left( \int\limits_{0}^{\infty} f \left( t \right) t^{\frac{p \eta_1 - H \left( c \right) - 1}{2}} \left( 1 - ct \right)_+^{\frac{p \eta_2 - H \left( c \right) - 1}{2}} \mathrm{d}t \right)^{\frac{1}{p}},
				\end{aligned}
			\end{equation}
			where, the constant $K$ depends only on $p, \eta_1, \eta_2$ and $c$.
		\end{lemma}
		\begin{proof}
			The proof of this result depends on the following inequality of \cite{KumarRay}. For $\gamma > 0$, $p \geq 1$, and characteristic functions of measurable subsets of $\left( 0, \infty \right)$ with finite measure, we have,
			\begin{equation}
				\label{I3}
				\int\limits_{0}^{\infty} f \left( s \right) s^{\gamma - 1} \mathrm{d}s \leq K \left( \int\limits_{0}^{\infty} f \left( s \right) s^{p \gamma - 1} \mathrm{d}s \right)^{\frac{1}{p}}.
			\end{equation}
			Here, $K$ is a constant that only depends on $p$ and $\gamma$. Equation \eqref{I3} is a special case of Equation \eqref{MainLemmaEquationUnified} for $c = 0$. We, therefore, prove the lemma for $c \neq 0$. It is easy to see by a simply change of variables that for characteristic functions of measurable subsets of $\left( 0, \infty \right)$, we also have, for $c \neq 0$,
			\begin{equation}
				\label{I4}
				\int\limits_{0}^{\infty} f \left( t \right) \left( 1 - ct \right)_+^{\gamma - 1} \mathrm{d}t \leq K \left( \int\limits_{0}^{\infty} f \left( t \right) \left( 1 - ct \right)_+^{p \gamma - 1} \mathrm{d}t \right)^{\frac{1}{p}}.
			\end{equation}
			Now, for $c \neq 0$, we consider the following.
			\begin{align*}
				\int\limits_{0}^{\infty} f \left( t \right) t^{\frac{\eta_1 - H \left( c \right) - 1}{2}} \left( 1 - ct \right)^{\frac{\eta_2 - H \left( c \right) - 1}{2}}_+ \mathrm{d}t &= \int\limits_{0}^{\frac{1}{2}} f \left( t \right) t^{\frac{\eta_1 - H \left( c \right) - 1}{2}} \left( 1 - ct \right)^{\frac{\eta_2 - H \left( c \right) - 1}{2}}_+ \mathrm{d}t \\
				&+ \int\limits_{\frac{1}{2}}^{1} f \left( t \right) t^{\frac{\eta_1 - H \left( c \right) - 1}{2}} \left( 1 - ct \right)^{\frac{\eta_2 - H \left( c \right) - 1}{2}}_+ \mathrm{d}t \\
				&+ \int\limits_{1}^{\infty} f \left( t \right) t^{\frac{\eta_1 - H \left( c \right) - 1}{2}} \left( 1 - ct \right)^{\frac{\eta_2 - H \left( c \right) - 1}{2}}_+ \mathrm{d}t.
			\end{align*}
			For the first integral, we observe that when $0 \leq t \leq \frac{1}{2}$, we have
			\begin{equation}
				\label{EstimateI1}
				A_1 \leq \left( 1 - ct \right)_+^{\frac{\eta_2 - H \left( c \right) - 1}{2}} \le A_2,
			\end{equation}
			for positive constant $A_1, A_2$ that depend only on $\eta_2$ and $c$. Similarly, when $\frac{1}{2} \leq t \leq 1$, we have
			\begin{equation}
				\label{EstimateI2}
				A_1' \leq t^{\frac{\eta_1 - H \left( c \right) - 1}{2}} \leq A_2',
			\end{equation}
			for positive constants $A_1'$ and $A_2'$ that depends only on $\eta_1$ and $c$. On the other hand, for $t \geq 1$, we have positive constants $B_1, B_2$ such that
			\begin{equation}
				\label{EstimateI3}
				B_1 \left( 1 - H \left( c \right) \right) \leq \frac{\left( 1 - ct \right)_+^{\frac{\eta_2 - H \left( c \right) - 1}{2}}}{t^{\frac{\left( 1 - H \left( c \right) \right) \left( \eta_2 - H \left( c \right) - 1 \right)}{2}}} \leq B_2 \left( 1 - H \left( c \right) \right).
			\end{equation}
			Thus, we get a constant $K_1 > 0$ such that
			\begin{align*}
				&\int\limits_{0}^{\infty} f \left( t \right) t^{\frac{\eta_1 - H \left( c \right) - 1}{2}} \left( 1 - ct \right)^{\frac{\eta_2 - H \left( c \right) - 1}{2}}_+ \mathrm{d}t \\
				&\leq K_1 \left[ \int\limits_{0}^{1} f \left( t \right) t^{\frac{\eta_1 - H \left( c \right) - 1}{2}} \mathrm{d}t + \int\limits_{\frac{1}{2}}^{1} f \left( t \right) \left( 1 - ct \right)^{\frac{\eta_2 - H \left( c \right) - 1}{2}} \mathrm{d}t + \left( 1 - H \left( c \right) \right) \int\limits_{1}^{\infty} f \left( t \right) t^{\frac{\eta_1 - H \left( c \right) - 1}{2} + \frac{\left( 1 - H \left( c \right) \right) \left( \eta_2 - H \left( c \right) - 1 \right)}{2}} \mathrm{d}t \right].
			\end{align*}
			Now, using Inequalities \eqref{I3} and \eqref{I4}, we get another constant $K_2 > 0$ such that
			\begin{align*}
				\int\limits_{0}^{\infty} f \left( t \right) t^{\frac{\eta_1 - H \left( c \right) - 1}{2}} \left( 1 - ct \right)^{\frac{\eta_2 - H \left( c \right) - 1}{2}}_+ \mathrm{d}t &\leq K_2 \left[ \left( \int\limits_{0}^{\frac{1}{2}} f \left( t \right) t^{\frac{p \left( \eta_1 + 1 - H \left( c \right) \right)}{2} - 1} \mathrm{d}t \right)^{\frac{1}{p}} + \left( \int\limits_{\frac{1}{2}}^{1} f \left( t \right) \left( 1 - ct \right)^{\frac{p \left( \eta_2 + 1 - H \left( c \right) \right)}{2} - 1} \mathrm{d}t \right)^{\frac{1}{p}} \right. \\
				&\left. + \left( \left( 1 - H \left( c \right) \right) \int\limits_{1}^{\infty} f \left( t \right) t^{\frac{p \left[ \eta_1 + 1 - H \left( c \right) + \left( 1 - H \left( c \right) \right) \left( \eta_2 - H \left( c \right) - 1 \right) \right]}{2} - 1} \mathrm{d}t \right)^{\frac{1}{p}} \right] \\
				&= K_2 \left[ \left( \int\limits_{0}^{\frac{1}{2}} f \left( t \right) t^{\frac{p\eta_1 - H \left( c \right) - 1}{2}} t^{\frac{\left( p - 1 \right) \left( 1 - H \left( c \right) \right)}{2}} \mathrm{d}t \right)^{\frac{1}{p}} \right. \\
				&\left. + \left( \int\limits_{\frac{1}{2}}^{1} f \left( t \right) \left( 1 - ct \right)^{\frac{p \eta_2 - H \left( c \right) - 1}{2}} \left( 1 - ct \right)^{\frac{\left( p - 1 \right) \left( 1 - H \left( c \right) \right)}{2}} \mathrm{d}t \right)^{\frac{1}{p}} \right. \\
				&\left. + \left( \left( 1 - H \left( c \right) \right) \int\limits_{1}^{\infty} f \left( t \right) t^{\frac{p \eta_1 - H \left( c \right) - 1}{2}} t^{\frac{\left( 1 - H \left( c \right) \right) \left( p \eta_2 - H \left( c \right) - 1 \right)}{2}} t^{- \frac{H \left( c \right) \left( p - 1 \right) \left( 1 - H \left( c \right) \right)}{2}} \mathrm{d}t \right)^{\frac{1}{p}} \right].
			\end{align*}
			Now, we notice that for $0 \leq t \leq \frac{1}{2}$, we have $t^{\frac{\left( p - 1 \right) \left( 1 - H \left( c \right) \right)}{2}} \leq M_1$ for some constant $M_1 > 0$. Similarly, for $\frac{1}{2} \leq t \leq 1$, there is some $M_2 > 0$ such that $\left( 1 - ct \right)^{\frac{\left( p - 1 \right) \left( 1 - H \left( c \right) \right)}{2}} \leq M_2$. Also, we notice that $H \left( c \right) \left( 1 - H \left( c \right) \right) = 0$ for any $c \in \mathbb{R}$. Hence, we get a constant $K_3 > 0$, depending only on $\eta_1, \eta_2, p$, and $c$, such that
			\begin{align*}
				\int\limits_{0}^{\infty} f \left( t \right) t^{\frac{\eta_1 - H \left( c \right) - 1}{2}} \left( 1 - ct \right)^{\frac{\eta_2 - H \left( c \right) - 1}{2}}_+ \mathrm{d}t &\leq K_3 \left[ \left( \int\limits_{0}^{\frac{1}{2}} f \left( t \right) t^{\frac{p\eta_1 - H \left( c \right) - 1}{2}} \mathrm{d}t \right)^{\frac{1}{p}} + \left( \int\limits_{\frac{1}{2}}^{1} f \left( t \right) \left( 1 - ct \right)^{\frac{p \eta_2 - H \left( c \right) - 1}{2}} \mathrm{d}t \right)^{\frac{1}{p}} \right. \\
				&\left. + \left( \left( 1 - H \left( c \right) \right) \int\limits_{1}^{\infty} f \left( t \right) t^{\frac{p \eta_1 - H \left( c \right) - 1}{2}} t^{\frac{\left( 1 - H \left( c \right) \right) \left( p \eta_2 - H \left( c \right) - 1 \right)}{2}} \mathrm{d}t \right)^{\frac{1}{p}} \right].
			\end{align*}
			Now, using the estimates given in Equations \eqref{EstimateI1}, \eqref{EstimateI2}, and \eqref{EstimateI3}, we get
			\begin{align*}
				&\int\limits_{0}^{\infty} f \left( t \right) t^{\frac{\eta_1 - H \left( c \right) - 1}{2}} \left( 1 - ct \right)^{\frac{\eta_2 - H \left( c \right) - 1}{2}}_+ \mathrm{d}t \leq K \left( \int\limits_{0}^{\infty} f \left( t \right) t^{\frac{p \eta_1 - H \left( c \right) - 1}{2}} \left( 1 - ct \right)_+^{\frac{p \eta_2 - H \left( c \right) - 1}{2}} \mathrm{d}t \right)^{\frac{1}{p}},
			\end{align*}
			where, the constant $K$ depends only on $p, \eta_1, \eta_2$, and $c$.
		\end{proof}
		The following corollary easily follows by replacing $t$ by $s_c^2 \left( \frac{t}{2} \right)$ and taking $\eta_1 = \eta_2 = \gamma$.
		\begin{corollary}
			\label{SCEstimate}
			Let $\gamma > 0$, and $f$ be the characteristic function of measurable subset of $\left( 0, \frac{\delta \left( X \right)}{2} \right)$. Then, we have for $p \geq 1$,
			\begin{equation}
				\label{LemmaEquation}
				\int\limits_{0}^{\frac{\delta \left( X_c \right)}{2}} f \left( t \right) s_c^{\gamma - H \left( c \right)} \left( t \right) \ \mathrm{d}t \leq K \left( \int\limits_{0}^{\frac{\delta \left( X_c \right)}{2}} f \left( t \right) s_c^{p \gamma - H \left( c \right)} \left( t \right) \ \mathrm{d}t \right)^{\frac{1}{p}},
			\end{equation}
			where $K$ is a constant that depends only on $p$, $c$, and $\gamma$, and $H$ is the heavyside function given by $H \left( s \right) = \begin{cases}
														1, & s \geq 0. \\
														0, & s < 0.
													\end{cases}$.
		\end{corollary}
		Now that we have all the tools we require, we present the main result of the article, which is the end-point estimate for the totally-geodesic $k$-plane transforms. The ``end-points" for $p$ arise due to existence conditions for the Euclidean space $\mathbb{R}^n$ (see \cite{Duoandikoetxea}) and real hyperbolic space $\mathbb{H}^n$ (see \cite{BerensteinRubin}). In these cases, the end-points are $p = \frac{n}{k}$ and $p = \frac{n - 1}{k - 1}$, respectively, for $\mathbb{R}^n$ and $\mathbb{H}^n$. For the case of the sphere, we do not have an end-point for the existence of the $k$-plane transform. In fact, in \cite{RubinInversion}, the authors prove that the $k$-plane transform is a bounded operator from $L^p \left( \mathbb{S}^n \right)$ to $L^p \left( \Xi_k \left( \mathbb{S}^n \right) \right)$, for all $p \geq 1$. Therefore, it is natural to ask about $L^p$-improving mapping properties of the $k$-plane transform. Particularly, one asks whether we have $L^p$-$L^{\infty}$ boundedness of the $k$-plane transform. It was proved, through an explicit example, in \cite{KumarRay} that with the natural measure on the domain and codomain, such a bound cannot be expected. In fact, the authors prove that we cannot have $L^{p, 1}-L^{\infty}$ boundedness as well, for any $p < + \infty$. However, if we add a cosine weight to the codomain ($\Xi_k \left( \mathbb{S}^n \right)$), analogous to the hyperbolic result of \cite{KumarRay}, the $L^{p, 1}-L^{\infty}$ inequality is expected only if $p \geq \frac{n}{k}$. This gives a natural end-point, namely $p = \frac{n}{k}$. We now state the unified end-point result for the totally-geodesic $k$-plane transform.
		\begin{theorem}
			\label{EndPointEstimate}
			Let $X$ be a simply connected space of constant curvature $c$, and let $f: X \rightarrow \mathbb{C}$ be a radial $L^{p, 1}$-function, where $L^{p, 1}$ is the Lorentz space, and $p = \frac{n}{k}$ for $X = \mathbb{R}^n$ and $\mathbb{S}^n$, and $p = \frac{n - 1}{k - 1}$, for $X = \mathbb{H}^n$. Then, we have for $2 \leq k \leq n - 1$,
			\begin{equation}
				\label{MainEquation}
				\| s_c' \left( d \left( o, \cdot \right) \right) R_kf \left( \cdot \right) \|_{L^{\infty} \left( \Xi_k \right)} \leq K \| f \|_{L^{p, 1} \left( X \right)}.
			\end{equation}
			Here, the constant $K$ depends only on $n$ and $k$.
		\end{theorem}
		\begin{proof}
			It is known that to prove Equation \eqref{MainEquation}, it is enough to consider characteristic functions of measurable radial sets of $X$ with finite measure (see \cite{SteinWeissFA}). Let $E$ be a measurable radial subset of $X$ of finite measure, and let $f$ be its characteristic function. Thus, there is a function $\tilde{f}$ such that $f \left( x \right) = \tilde{f} \left( s_c \left( \frac{d \left( o, x \right)}{2} \right) \right)$ (see Definition \ref{RadialFunctions}). Then, we have by the polar decomposition of $X$,
			$$\| f \|_{L^{p, 1} \left( X \right)} = \mu \left( E \right)^{\frac{1}{p}} = K \left( \int\limits_{0}^{\frac{\delta \left( X \right)}{2}} \tilde{f} \left( s_c \left( \frac{r}{2} \right) \right) s_c^{n - 1} \left( r \right) \mathrm{d}r \right)^{\frac{1}{p}}.$$
			Since $f$ is a characteristic function of a radial set, $\tilde{f} \circ s_c$ is also a characteristic function of a measurable subset of $\left[ 0, \frac{\delta \left( X \right)}{2} \right]$. Now, we can use Corollary \ref{SCEstimate} in our proof. To begin, we first consider the function $G \left( t \right) = \left( \ln s_c \right)' \left( t \right) = \frac{s_c' \left( t \right)}{s_c \left( t \right)}$. Then,
			$$G' \left( t \right) = \dfrac{s_c \left( t \right) s_c'' \left( t \right) - \left( s_c' \right)^2 \left( t \right)}{s_c^2 \left( t \right)}.$$
			Using Equations \eqref{CurvatureDE} and \eqref{ScNonLinDE}, we have $s_c s_c'' - \left( s_c' \right)^2 = - \left( cs_c + \left( s_c' \right)^2 \right) = -1$. Thus, $G$ is a decreasing function, and hence for $t \geq d \left( o, \xi \right)$, we have $\left( \ln s_c \right)' \left( t \right) \leq \left( \ln s_c \right)' \left( d \left( o, \xi \right) \right)$. Consequently, we get
			$$0 \leq \left[ 1 - \left( \dfrac{\left( \ln s_c \right)' \left( t \right)}{\left( \ln s_c \right)' \left( d \left( o, \xi \right) \right)} \right)^2 \right]^{\frac{k}{2} - 1} \leq 1,$$
			for $k \geq 2$.	Now, using Theorem \ref{TGRTRadialFunction} and Corollary \ref{SCEstimate}, we have
			\begin{align*}
				\left| s_c' \left( d \left( o, \xi \right) \right) R_kf \left( \xi \right) \right| &= \left| K \int\limits_{d \left( o, \xi \right)}^{\frac{\delta \left( X \right)}{2}} \tilde{f} \left( s_c \left( \frac{t}{2} \right) \right) \left[ 1 - \left( \dfrac{\left( \ln s_c \right)' \left( t \right)}{\left( \ln s_c \right)' \left( d \left( o, \xi \right) \right)} \right)^{2} \right]^{\frac{k}{2}} s_c^{k - 1} \left( t \right) \mathrm{d}t \right| \\
				&\leq K \int\limits_{0}^{\frac{\delta \left( X \right)}{2}} \tilde{f} \left( s_c \left( \frac{t}{2} \right) \right) s_c^{k - 1} \left( t \right) \mathrm{d}t \\
				&\leq K \left( \int\limits_{0}^{\frac{\delta \left( X \right)}{2}} \tilde{f} \left( s_c \left( \frac{t}{2} \right) \right) s_c^{n - 1} \left( t \right) \mathrm{d}t \right)^{\frac{1}{p}} = K \| f \|_{L^{p, 1} \left( X \right)}.
			\end{align*}
		\end{proof}
		We now give a unified proof for the case $k = 1$. We have mentioned in Remark \ref{RemarkHnEndPoint} that this end-point can arise only for $c = 0$ or $c = 1$.
		\begin{theorem}
			\label{XRayEndPointUnified}
			Let $X$ be a simply connected space of constant curvature $c$, where $c = 0$ or $c = 1$. Let $f: X \rightarrow \mathbb{C}$ be a radial $L^{n, 1}$-function. Then, we have,
			\begin{equation}
				\label{XRAyEndPointBoundedness}
				\| s_c' \left( d \left( o, \cdot \right) \right) R_1f \|_{\infty} \leq K \| f \|_{L^{n, 1} \left( X \right)}.
			\end{equation}
		\end{theorem}
		\begin{proof}
			From Equation \eqref{RadialFunctionEquation}, we have,
			\begin{align*}
				&s_c' \left( d \left( o, \xi \right) \right) R_1f \left( \xi \right) \\
				&= K \int\limits_{d \left( o, \xi \right)}^{\frac{\delta \left( X \right)}{2}} \tilde{f} \left( s_c \left( \frac{t}{2} \right) \right) \left[ 1 - \left( \frac{\left( \ln s_c \right)' \left( t \right)}{\left( \ln s_c \right)' \left( d \left( o, \xi \right) \right)} \right) \right]^{- \frac{1}{2}} \mathrm{d}t \\
				&= K s_c' \left( d \left( o, \xi \right) \right) \times \\
				&\int\limits_{d \left( o, \xi \right)}^{\frac{\delta \left( X \right)}{2}} \tilde{f} \left( s_c \left( \frac{t}{2} \right) \right) \left[ \left( s_c \left( t \right) \right)^2 \left( s_c' \left( d \left( o, \xi \right) \right) \right)^2 - \left( s_c' \left( t \right) \right)^2 \left( s_c \left( d \left( o, \xi \right) \right) \right)^2 \right]^{- \frac{1}{2}} s_c \left( t \right) \mathrm{d}t.
			\end{align*}
			By using Equation \eqref{ScNonLinDE}, and simplifying, we get
			\begin{align*}
				s_c' \left( d \left( o, \xi \right) \right) R_1f \left( \xi \right) &= K s_c' \left( d \left( o, \xi \right) \right) \int\limits_{d \left( o, \xi \right)}^{\frac{\delta \left( X \right)}{2}} \tilde{f} \left( s_c \left( \frac{t}{2} \right) \right) \left( s_c^2 \left( t \right) - s_c^2 d \left( o, \xi \right) \right)^{- \frac{1}{2}} s_c \left( t \right) \mathrm{d}t.
			\end{align*}
			Now, let us substitute $s_c \left( \frac{t}{2} \right) = x \ s_c \left( \frac{d \left( o, \xi \right)}{2} \right)$. Then, we would have $s_c' \left( \frac{t}{2} \right) \mathrm{d}t = 2 s_c \left( \frac{d \left( o, \xi \right)}{2} \right) \mathrm{d}x$. When $t = d \left( o, \xi \right)$, we have $x = 1$; and when $t = \frac{\delta \left( X \right)}{2}$, we have $x = \frac{s_c \left( \frac{\delta \left( X \right)}{4} \right)}{s_c \left( \frac{d \left( o, \xi \right)}{2} \right)}$. To make the notation less cumbersome, we use $\gamma = \frac{s_c \left( \frac{\delta \left( X \right)}{4} \right)}{s_c \left( \frac{d \left( o, \xi \right)}{2} \right)}$. Using the double angle formula (Equation \eqref{DoubleAngleFormula1}), we get
			\begin{align*}
				s_c' \left( d \left( o, \xi \right) \right) R_1f \left( \xi \right) &= K s_c' \left( d \left( o, \xi \right) \right) \int\limits_{1}^{\gamma} \tilde{f} \left( x \ s_c \left( \frac{d \left( o, \xi \right)}{2} \right) \right) \left[ 4x^2 s_c^2 \left( \frac{d \left( o, \xi \right)}{2} \right) \left( 1 - cx^2 s_c^2 \left( \frac{d \left( o, \xi \right)}{2} \right) \right) \right. \\
				&\left. - 4s_c^2 \left( \frac{d \left( o, \xi \right)}{2} \right) \left( 1 - cs_c^2 \left( \frac{d \left( o, \xi \right)}{2} \right) \right) \right]^{- \frac{1}{2}} x \ s_c^2 \left( \frac{d \left( o, \xi \right)}{2} \right) \mathrm{d}x \\
				&= K s_c' \left( d \left( o, \xi \right) \right) s_c \left( \frac{d \left( o, \xi \right)}{2} \right) \int\limits_{1}^{\gamma} \tilde{f} \left( x \ s_c \left( \frac{d \left( o, \xi \right)}{2} \right) \right) \left[ \left( x^2 - 1 \right) - cs_c^2 \left( \frac{d \left( o, \xi \right)}{2} \right) \left( x^4 - 1 \right) \right]^{- \frac{1}{2}} x \ \mathrm{d}x \\
				&= K s_c' \left( d \left( o, \xi \right) \right) s_c \left( \frac{d \left( o, \xi \right)}{2} \right) \int\limits_{1}^{\gamma} \tilde{f} \left( x \ s_c \left( \frac{d \left( o, \xi \right)}{2} \right) \right) \left( x^2 - 1 \right)^{- \frac{1}{2}} \left[ 1 - \left( x^2 + 1 \right) cs_c^2 \left( \frac{d \left( o, \xi \right)}{2} \right) \right]^{- \frac{1}{2}} x \ \mathrm{d}x.
			\end{align*}
			Using Equation \eqref{DoubleAngleFormula3}, we have,
			\begin{align*}
				s_c' \left( d \left( o, \xi \right) \right) R_1f \left( \xi \right) &= K s_c' \left( d \left( o, \xi \right) \right) s_c \left( \frac{d \left( o, \xi \right)}{2} \right) \int\limits_{1}^{\gamma} \tilde{f} \left( x \ s_c \left( \frac{d \left( o, \xi \right)}{2} \right) \right) \left( x^2 - 1 \right)^{- \frac{1}{2}} \times \\
				&\left[ s_c' \left( d \left( o, \xi \right) \right) - \left( x^2 - 1 \right) cs_c^2 \left( \frac{d \left( o, \xi \right)}{2} \right) \right]^{- \frac{1}{2}} x \ \mathrm{d}x \\
				&= K \left( s_c' \right)^{\frac{1}{2}} \left( d \left( o, \xi \right) \right) s_c \left( \frac{d \left( o, \xi \right)}{2} \right) \int\limits_{1}^{\gamma} \tilde{f} \left( x \ s_c \left( \frac{d \left( o, \xi \right)}{2} \right) \right) \left( x^2 - 1 \right)^{- \frac{1}{2}} \times \\
				&\left[ 1 - c \left( x^2 - 1 \right) \frac{s_c^2 \left( \frac{d \left( o, \xi \right)}{2} \right)}{s_c' \left( d \left( o, \xi \right) \right)} \right]^{- \frac{1}{2}} x \ \mathrm{d}x.
			\end{align*}
			To make the notation less cumbersome, let us write $\alpha = s_c \left( \frac{d \left( o, \xi \right)}{2} \right)$, $\beta = s_c' \left( d \left( o, \xi \right) \right)$, $\mu = s_c \left( \frac{\delta \left( X \right)}{4} \right)$. That is,
			\begin{align*}
				s_c' \left( d \left( o, \xi \right) \right) R_1f \left( \xi \right) &= K \alpha \beta^{\frac{1}{2}} \int\limits_{1}^{\gamma} \tilde{f} \left( \alpha x \right) \left( x^2 - 1 \right)^{- \frac{1}{2}} \left[ 1 - c \frac{\alpha^2}{\beta} \left( x^2 - 1 \right) \right]^{- \frac{1}{2}} x \mathrm{d}x.
			\end{align*}
			Upon substituting $x^2 - 1 = \frac{\beta}{\alpha^2} t$, we get
			\begin{align*}
				s_c' \left( d \left( o, \xi \right) \right) R_1f \left( \xi \right) \leq K \beta \int\limits_{0}^{\zeta} \tilde{f} \left( \alpha \sqrt{\frac{\beta}{\alpha^2} t + 1} \right) t^{- \frac{1}{2}} \left[ 1 - ct \right]^{- \frac{1}{2}} \mathrm{d}t,
			\end{align*}
			where, $\zeta = \frac{\alpha^2}{\beta} \left( \frac{\mu^2}{\alpha^2} - 1 \right)$.
			Now, using Lemma \ref{MainLemmaUnified} with $p = n$ and $\eta_1 = \eta_2 = 1$, we get
			\begin{align*}
				s_c' \left( d \left( o, \xi \right) \right) R_1f \left( \xi \right) \leq K \beta \left[ \int\limits_{0}^{\zeta} \tilde{f} \left( \alpha \sqrt{\frac{\beta}{\alpha^2} t + 1} \right) t^{\frac{n}{2} - 1} \left( 1 - ct \right)^{\frac{n}{2} - 1} \mathrm{d}t \right]^{\frac{1}{n}}.
			\end{align*}
			Now, we substitute $t = \frac{\alpha^2}{\beta} u^2$ to get
			\begin{align*}
				&s_c' \left( d \left( o, \xi \right) \right) R_1f \left( \xi \right) \leq K \alpha \beta^{\frac{1}{2}} \left[ \int\limits_{0}^{\nu} \tilde{f} \left( \alpha \sqrt{u^2 + 1} \right) u^{n - 1} \left( 1 - cu^2 \frac{\alpha^2}{\beta} \right)^{\frac{n}{2} - 1} \mathrm{d}u \right]^{\frac{1}{n}},
			\end{align*}
			where, $\nu = \left( \frac{\mu^2}{\alpha^2} - 1 \right)^{\frac{1}{2}}$. Now, we have,
			\begin{align*}
				s_c' \left( d \left( o, \xi \right) \right) R_1f \left( \xi \right) &\leq K \alpha \beta^{\frac{1}{2} - \frac{1}{n}} \left[ \int\limits_{0}^{\nu} \tilde{f} \left( \alpha \sqrt{u^2 + 1} \right) \left( u^2 + 1 \right)^{\frac{n}{2} - 1} \left( 1 - cu^2 \frac{\alpha^2}{\beta} \right)^{\frac{n}{2} - 1} u \ \mathrm{d}u \right]^{\frac{1}{n}} \\
				&= K \alpha \beta^{\frac{1}{n}} \left[ \int\limits_{0}^{\nu} \tilde{f} \left( \alpha \sqrt{u^2 + 1} \right) \left( u^2 + 1 \right)^{\frac{n}{2} - 1} \left( \beta - c \alpha^2 u^2 \right)^{\frac{n}{2} - 1} u \ \mathrm{d}u \right]^{\frac{1}{n}} \\
				&\leq K \alpha \left[ \int\limits_{0}^{\nu} \tilde{f} \left( \alpha \sqrt{u^2 + 1} \right) \left( u^2 + 1 \right)^{\frac{n}{2} - 1} \left( \beta - c \alpha^2 u^2 \right)^{\frac{n}{2} - 1} u \ \mathrm{d}u \right]^{\frac{1}{n}},
			\end{align*}
			since $\beta = s_c' \left( d \left( o, \xi \right) \right) \leq 1$.
			Resubstituting the values of $\alpha$ and $\beta$, and using Equation \eqref{ScNonLinDE}, we get
			\begin{align*}
				&s_c' \left( d \left( o, \xi \right) \right) R_1f \left( \xi \right) \\
				&\leq K s_c \left( \frac{d \left( o, \xi \right)}{2} \right) \left[ \int\limits_{0}^{\nu} \tilde{f} \left( \sqrt{u^2 + 1} s_c \left( \frac{d \left( o, \xi \right)}{2} \right) \right) \left( u^2 + 1 \right)^{\frac{n}{2} - 1} \left( 1 - c \left( u^2 + 2 \right) s_c^2 \left( \frac{d \left( o, \xi \right)}{2} \right) \right)^{\frac{n}{2} - 1} u \ \mathrm{d}u \right]^{\frac{1}{n}} \\
				&\leq K s_c \left( \frac{d \left( o, \xi \right)}{2} \right) \left[ \int\limits_{0}^{\nu} \tilde{f} \left( \sqrt{u^2 + 1} s_c \left( \frac{d \left( o, \xi \right)}{2} \right) \right) \left( u^2 + 1 \right)^{\frac{n}{2} - 1} \left( 1 - c \left( u^2 + 1 \right) s_c^2 \left( \frac{d \left( o, \xi \right)}{2} \right) \right)^{\frac{n}{2} - 1} u \ \mathrm{d}u \right]^{\frac{1}{n}} .
			\end{align*}
			Finally, by substituting $s_c \left( \frac{t}{2} \right) = \sqrt{u^2 + 1} s_c \left( \frac{d \left( o, \xi \right)}{2} \right)$, and simplifying, we obtain
			\begin{align*}
				s_c' \left( d \left( o, \xi \right) \right) R_1f \left( \xi \right) &\leq K \left[ \int\limits_{d \left( o, \xi \right)}^{\frac{\delta \left( X \right)}{2}} \tilde{f} \left( s_c \left( \frac{t}{2} \right) \right) s_c^{n - 1} \left( t \right) \mathrm{d}t \right]^{\frac{1}{n}} \leq K \| f \|_{n, 1}.
			\end{align*}
			This completes the proof!
		\end{proof}
	\section{Some inequalities for Hypergeometric Functions}
		\label{SpecialFunctionSection}
			We take a short detour for the theme of this article and comment upon a consequence of Lemma \ref{MainLemmaUnified} about hypergeometric functions. We first give certain definitions. These definitions can be found in \cite{Olver}. First, we give the definition of Appell's hypergeometric function of the first kind.
			\begin{equation}
				\label{Appell1Function}
				F_1 \left( \alpha; \beta_1, \beta_2; \gamma; x, y \right) = \frac{\Gamma \left( \gamma \right)}{\Gamma \left( \alpha \right) \Gamma \left( \gamma - \alpha \right)} \int\limits_{0}^{1} \frac{u^{\alpha - 1} \left( 1 - u \right)^{\gamma - \alpha - 1}}{\left( 1 - ux \right)^{\beta_1} \left( 1 - uy \right)^{\beta_2}} \mathrm{d}u,
			\end{equation}
			where, $\text{Re } \alpha > 0$, $\text{Re} \left( \gamma - \alpha \right)  > 0$, and $x, y \notin \left[ 1, \infty \right)$.
			
			We also require the following transformation of variable in Appell's function (see \cite{Olver}),
			\begin{equation}
				\label{AppellTransformationVariable}
				F_1 \left( \alpha; \beta_1; \beta_2; \gamma; x, y \right) = \left( 1 - x \right)^{- \alpha} F_1 \left( \alpha; \gamma - \beta_1 - \beta_2; \beta_2; \gamma; \frac{x}{x - 1}; \frac{y - x}{1 - x} \right).
			\end{equation}
			
			On the other hand, the Gauss hypergeometric function is given by
			\begin{equation}
				\label{Definition2F1}
				{}_2F_1 \left( \alpha, \beta; \gamma; z \right) = \frac{\Gamma \left( \gamma \right)}{\Gamma \left( \beta \right) \Gamma \left( \gamma - \beta \right)} \int\limits_{0}^{1} \frac{u^{\beta - 1} \left( 1 - u \right)^{\gamma - \beta - 1}}{\left( 1 - zu \right)^{\alpha}} \mathrm{d}u,
			\end{equation}
			where, $\text{Re } \gamma > \text{Re } \beta > 0$ and $z \notin \left[ 1, \infty \right)$.
			
			The following useful transformation of variables for hypergeometric function is available in \cite{Olver}.
			\begin{equation}
				\label{HypergeometricTransformationVariable}
				{}_2F_1 \left( \alpha, \beta; \gamma; z \right) = \left( 1 - z \right)^{- \alpha} {}_2F_1 \left( \alpha, \gamma - \beta; \gamma; \frac{z}{z - 1} \right),
			\end{equation}
			for $z \notin \left[ 1, \infty \right)$.
			
			With these definitions made, we give an inequality about Appell's hypergeometric function of the first kind.
			\begin{theorem}
				\label{AppellInequalities}
				For $c \in \left\lbrace -1, 0, 1 \right\rbrace$ and $0 < a < b < \frac{1}{1 - H \left( -c \right)}$, we have,
				\begin{equation}
					\label{ANotZeroUnified}
					\begin{aligned}
						&\left( 1 - ca \right)^{- \frac{H \left( c \right) + 1}{2p'}} a^{\frac{1 - H \left( c \right)}{2p'}} \left( 1 - \frac{a}{b} \right)^{\frac{1}{p'}} F_1 \left( 1; 1 + \frac{\eta_1 + \eta_2}{2} - H \left( c \right), \frac{H \left( c \right) + 1 - \eta_2}{2}; 2; 1 - \frac{a}{b}, \frac{b - a}{\left( 1 - ca \right) b} \right) \\
						&\leq K \left[ F_1 \left( 1; 1 - H \left( c \right) + p \left( \frac{\eta_1 + \eta_2}{2}, \frac{H \left( c \right) + 1 - p \eta_2}{2}; 2; 1 - \frac{a}{b}, \frac{b - a}{\left( 1 - ca \right) b} \right) \right) \right]^{\frac{1}{p}}.
					\end{aligned}
				\end{equation}
				As special cases, we have for $0 < a < b < \infty$,
				\begin{equation}
					\label{ANotZeroCMinus1}
					\begin{aligned}
						&\left( \frac{a}{1 + a} \right)^{\frac{1}{2p'}} \left( 1 - \frac{a}{b} \right)^{\frac{1}{p'}} F_1 \left( 1; 1 + \frac{\eta_1 + \eta_2}{2}, 1 - \frac{\eta_2}{2}; 2; 1 - \frac{a}{b}, \frac{b - a}{\left( 1  + a \right) b} \right) \\
						&\leq K \left[ F_1 \left( 1; 1 - p \left( \frac{\eta_1 + \eta_2}{2} \right), \frac{1 - p \eta_2}{2}; 2; 1 - \frac{a}{b}, \frac{b - a}{\left( 1 + a \right) b} \right) \right]^{\frac{1}{p}},
					\end{aligned}
				\end{equation}
				\begin{equation}
					\label{ANotZeroCZero1}
					\begin{aligned}
						&\left( 1 - \frac{a}{b} \right)^{\frac{1}{p'}} F_1 \left( 1; \frac{\eta_1 + \eta_2}{2}, 1 - \frac{\eta_2}{2}; 2; 1 - \frac{a}{b}, 1 - \frac{a}{b} \right) \leq K \left[ F_1 \left( 1; p \left( \frac{\eta_1 + \eta_2}{2} \right), 1 - \frac{p \eta_2}{2}; 2; 1 - \frac{a}{b}, 1 - \frac{a}{b} \right) \right]^{\frac{1}{p}},
					\end{aligned}
				\end{equation}
				and
				\begin{equation}
					\label{ANotZeroCZero2}
					\begin{aligned}
						a^{- \frac{1}{p'}} \left( 1 - \frac{a}{b} \right)^{\frac{1}{p'}} {}_2F_1 \left( 1 + \frac{\eta_1}{2}, 1; 2; 1 - \frac{a}{b} \right) \leq K \left[ {}_2F_1 \left( 1 + \frac{p \eta_1}{2}, 1; 2; 1 - \frac{a}{b} \right) \right]^{\frac{1}{p}}.
					\end{aligned}
				\end{equation}
				Also, for $0 < a < b < 1$, we have,
				\begin{equation}
					\label{ANotZeroC1}
					\begin{aligned}
						&\left( 1 - a \right)^{- \frac{1}{p'}} \left( 1 - \frac{a}{b} \right)^{\frac{1}{p'}} F_1 \left( 1; \frac{\eta_1 + \eta_2}{2}, 1 - \frac{\eta_2}{2}; 1 - \frac{a}{b}, \frac{b - a}{\left( 1 - a \right) b} \right) \leq K \left[ F_1 \left( 1; p \left( \frac{\eta_1 + \eta_2}{2} \right), 1 - \frac{p \eta_2}{2}; 2; 1 - \frac{a}{b}, \frac{b - a}{\left( 1 - a \right) b} \right) \right]^{\frac{1}{p}}.
					\end{aligned}
				\end{equation}
			\end{theorem}
			\begin{proof}
				Let us fix $0 < a < b \leq \frac{1}{1 - H \left( -c \right)}$, and consider $f = \chi_{\left( a, b \right)}$. Then, from Inequality \eqref{MainLemmaEquationUnified}, we get
			\begin{equation}
				\label{SpecialInequality}
				\int\limits_{a}^{b} t^{\frac{\eta_1 - H \left( c \right) - 1}{2}} \left( 1 - ct \right)^{\frac{\eta_2 - H \left( c \right) - 1}{2}} \mathrm{d}t \leq K \left( \int\limits_{a}^{b} t^{\frac{p \eta_1 - H \left( c \right) - 1}{2}} \left( 1 - ct \right)^{\frac{p \eta_2 - H \left( c \right) - 1}{2}} \mathrm{d}t \right)^{\frac{1}{p}}.
			\end{equation}
			By substituting $t = a + \left( b - a \right) v$ on both sides, and simplifying, we get
			\begin{equation}
				\label{IntermediateEstimateF12F1}
				\begin{aligned}
					&\left( 1 - ca \right)^{\frac{\eta_2 - H \left( c \right) - 1}{2}} a^{\frac{\eta_1 - H \left( c \right) - 1}{2}} \left( b - a \right) \int\limits_{0}^{1} \left( 1 - \left( 1 - \frac{b}{a} \right) v \right)^{\frac{\eta_1 - H \left( c \right) - 1}{2}} \left( 1 - \frac{c \left( b - a \right)}{\left( 1 - ca \right) b} v \right)^{\frac{\eta_2 - H \left( c \right) - 1}{2}} \mathrm{d}v \\
					&\leq K \left[ \left( 1 - ca \right)^{\frac{p \eta_2 - H \left( c \right) - 1}{2}} a^{\frac{p \eta_1 - H \left( c \right) - 1}{2}} \left( b - a \right) \int\limits_{0}^{1} \left( 1 - \left( 1 - \frac{b}{a} \right) v \right)^{\frac{p \eta_1 - H \left( c \right) - 1}{2}} \left( 1 - \frac{c \left( b - a \right)}{\left( 1 - ca \right) b} v \right)^{\frac{p \eta_2 - H \left( c \right) - 1}{2}} \mathrm{d}v \right]^{\frac{1}{p}}.
				\end{aligned}
			\end{equation}
			Now, from Equation \eqref{Appell1Function}, we get
			\begin{align*}
				&\left( 1 - ca \right)^{\frac{\eta_2 - H \left( c \right) - 1}{2}} a^{\frac{\eta_1 - H \left( c \right) + 1}{2}} \left( b - a \right) F_1 \left( 1; \frac{H \left( c \right) + 1 - \eta_1}{2}, \frac{H \left( c \right) + 1 - \eta_2}{2}; 2; 1 - \frac{b}{a}, \frac{c \left( b - a \right)}{\left( 1 - ca \right) b} \right) \\
				&\leq K \left[ \left( 1 - ca \right)^{\frac{p \eta_2 - H \left( c \right) - 1}{2}} a^{\frac{p \eta_1 - H \left( c \right) - 1}{2}} \left( b - a \right) F_1 \left( 1; \frac{H \left( c \right) + 1 - p \eta_1}{2}, \frac{H \left( c \right) + 1 - p \eta_2}{2}; 2; 1 - \frac{b}{a}; \frac{c \left( b - a \right)}{\left( 1 - ca \right) b} \right) \right]^{\frac{1}{p}}.
			\end{align*}
			Equation \eqref{ANotZeroUnified} follows easily from Equation \eqref{AppellTransformationVariable} and some simplification. To obtain Equation \eqref{ANotZeroCMinus1}, we consider $c = -1$. Equation \eqref{ANotZeroCZero1} is obtained from Equation \eqref{ANotZeroUnified} by considering $c = 0$. Further, by using Equations \eqref{Appell1Function} and \eqref{Definition2F1}, we get Equation \eqref{ANotZeroCZero2}. Lastly, Equation \eqref{ANotZeroC1} is obtained by using $c = 1$ in Equation \eqref{ANotZeroUnified}.
			\end{proof}
			Now, we give some inequalities for Gauss' hypergeometric function given in Equation \eqref{Definition2F1}.
			\begin{theorem}
				\label{GaussInequalities}
				For $c \in \left\lbrace -1, 0, 1 \right\rbrace$ and $0 < b < \frac{1}{1 - H \left( -c \right)}$, we have,
				\begin{equation}
					\label{AZeroUnified}
					\begin{aligned}
						&b^{\frac{1 - H \left( c \right)}{2p'}} \left( \frac{\eta_1 - H \left( c \right) - 1}{2} \right) {}_2F_1 \left( \frac{1 + H \left( c \right) - \eta_2}{2}, \frac{\eta_1 - H \left( c \right)+ 1}{2}; \frac{\eta_1 - H \left( c \right) + 3}{2}; cb \right) \\
						&\leq K \left( \frac{p \eta_1 - H \left( c \right) + 1}{2} \right)^{\frac{1}{p}} \left[ {}_2F_1 \left( \frac{1 + H \left( c \right) - p \eta_2}{2}, \frac{p \eta_1 - H \left( c \right) + 1}{2}; \frac{p \eta_1 - H \left( c \right) + 3}{2}; cb \right) \right]^{\frac{1}{p}}.
					\end{aligned}
				\end{equation}
				As special cases, we have for $b > 0$,
				\begin{equation}
				\label{AZeroCminus12}
					\begin{aligned}
						&\left( \frac{b}{1 + b} \right)^{\frac{1}{2p'}} \left( \frac{\eta_2 - 1}{2} \right) {}_2F_1 \left( \frac{1 - \eta_2}{2}, 1; \frac{\eta_1 + 3}{2}; \frac{b}{1 + b} \right) \leq K \left( \frac{p \eta_1 + 1}{2} \right)^{\frac{1}{p}} \left[ {}_2F_1 \left( 1 - \frac{p \eta_2}{2}, 1; \frac{p \eta_1 + 3}{2}; \frac{b}{1 + b} \right) \right]^{\frac{1}{p}}.
					\end{aligned}
				\end{equation}
				Also, for $0 < b < 1$, we have,
				\begin{equation}
					\label{AZeroC1}
					\begin{aligned}
						\left( \frac{\eta_1}{2} - 1 \right) {}_2F_1 \left( 1 - \frac{\eta_2}{2}, \frac{\eta_1}{2}; 1 + \frac{\eta_1}{2}; b \right) \leq K \left( \frac{p \eta_1}{2} \right)^{\frac{1}{p}} \left[ {}_2F_1 \left( 1 - \frac{p \eta_2}{2}, \frac{p \eta_1}{2}; 1 + \frac{p \eta_1}{2}; b \right) \right]^{\frac{1}{p}}
					\end{aligned}
				\end{equation}
			\end{theorem}
			\begin{proof}
				Similar to the proof of Theorem \ref{AppellInequalities}, we consider $f = \chi_{\left( 0, b \right)}$, for $0 < b < \frac{1}{1 - H \left( -c \right)}$. Then, Inequality \eqref{SpecialInequality} becomes (by taking $a = 0$),
				\begin{equation}
					\label{SpecialInequality2}
					\int\limits_{0}^{b} t^{\frac{\eta_1 - H \left( c \right) - 1}{2}} \left( 1 - ct \right)^{\frac{\eta_2 - H \left( c \right) - 1}{2}} \mathrm{d}t \leq K \left( \int\limits_{0}^{b} t^{\frac{p \eta_1 - H \left( c \right) - 1}{2}} \left( 1 - ct \right)^{\frac{p \eta_2 - H \left( c \right) - 1}{2}} \mathrm{d}t \right)^{\frac{1}{p}}.
				\end{equation}
				Next, we substitute $t = bv$ and simplify to obtain
				\begin{equation}
					\label{IntermediateEstimateF12F12}
					\begin{aligned}
						&b^{\frac{\eta_1 - H \left( c \right) + 1}{2}} \int\limits_{0}^{1} v^{\frac{\eta_1 - H \left( c \right) - 1}{2}} \left( 1 - cbv \right)^{\frac{\eta_2 - H \left( c \right) - 1}{2}} \mathrm{d}v \leq K \left[ b^{\frac{p \eta_1 - H \left( c \right) + 1}{2}} \int\limits_{0}^{1} v^{\frac{p \eta_1 - H \left( c \right) - 1}{2}} \left( 1 - cbv \right)^{\frac{p \eta_2 - H \left( c \right) - 1}{2}} \mathrm{d}v \right]^{\frac{1}{p}}.
					\end{aligned}
				\end{equation}
				Using the integral form of hypergometric function given in Equation \eqref{Definition2F1}, we get Equation \eqref{AZeroUnified}.
				
				Now, we notice that for $c = -1$, we get
				\begin{equation}
					\label{AZeroCMinus11}
					\begin{aligned}
						&b^{\frac{1}{2p'}} \left( \frac{\eta_2 - 1}{2} \right) {}_2F_1 \left( \frac{1 - \eta_2}{2}, \frac{\eta_1 - 1}{2}; \frac{\eta_1 + 3}{2}; -b \right) \leq K \left( \frac{p \eta_1 + 1}{2} \right)^{\frac{1}{p}} \left[ {}_2F_1 \left( \frac{1 - p \eta_2}{2}, \frac{p \eta_1 + 1}{2}; \frac{p \eta_1 + 3}{2}; -b \right) \right]^{\frac{1}{p}}.
					\end{aligned}
				\end{equation}
				Equation \eqref{AZeroCminus12} follows from an application of Equation \eqref{HypergeometricTransformationVariable}. Equation \eqref{AZeroC1} is obtained by using $c = 1$ in Equation \eqref{AZeroUnified}.
			\end{proof}
	\section{Conclusion}
		\label{ConclusionSection}
		In this article, we have seen a differential geometry perspective to obtain mapping properties of the $k$-dimensional totally-geodesic Radon transform. We have studied the simply connected spaces of constant curvature, which are necessarily isometric to one of the model spaces ($\mathbb{R}^n$, $\mathbb{H}^n$, or $\mathbb{S}^n$), where we have seen the naturality of the end-point estimates given in \cite{Duoandikoetxea} and \cite{KumarRay} depending on the curvature. This ``naturality" helps us unify the proofs of the end-point estimates. We comment here that the important link in unifying the end-point estimates was a unified Pythagoras' theorem. We emphasize here that at the moment, we do not claim to have found the best possible constants. It remains open to get the sharp bounds for the end-point estimates on all three spaces of constant curvature.
		
		The crux of the unified proof lied in Lemma \ref{MainLemmaUnified}, which has an application to obtaining certain inequalities of special functions. Particularly, we have seen some inequalities concerning Appell's function of the first kind and Gauss' hypergeometric function. We would like to mention here that these functions come up over and over when studying the spaces of constant curvature. The simplest example of their occurrence would be in evaluating the measure of a geodesic ball of finite radius. The measure can be expressed in terms of Gauss' hypergeomtric function. On the other hand, the measure of an annulus in the hyperbolic space (or the sphere) can be evaluated in terms of Appell's function of first kind. These observations warrant the need for a study linking various properties of hypergeometric functions to their applications in understanding the mapping properties of the $k$-plane transform.
		
		It also remains to be seen whether one can obtain weighted end-point estimates and weighted $L^p$-$L^q$ estimates for the $k$-plane transform on spaces of constant curvature. On the Euclidean space, this work was carried out by Kumar and Ray in \cite{KumarRayWE}. It would be interesting to get necessary and sufficient conditions for weighted boundedness of the totally-geodesic $k$-plane transform on the real hyperbolic space and the sphere.

\end{document}